\pdfoutput=1

\documentclass[11pt]{article}

\usepackage[margin=1cm]{geometry}
\usepackage{fullpage}
\usepackage{amsmath}
\usepackage{accents}
\usepackage{amsthm}
\usepackage{amssymb}
\usepackage[hidelinks]{hyperref}
\usepackage{cleveref}
\usepackage{textcomp}
\usepackage{enumerate}

\newcommand{\mbb}[1]{\mathbb{#1}}
\newcommand{\mbf}[1]{\mathbf{#1}}
\newcommand{\bs}{\boldsymbol}
\newcommand{\tr}{\textup{tr}\,}
\newcommand{\wt}{\widetilde}
\newcommand{\wh}{\widehat}
\newcommand{\mc}{\mathcal}

\newcommand{\Tr}[1]{\left\langle#1\right\rangle}
\newcommand{\pf}{\textup{pf}\,}
\renewcommand{\det}{\textup{det}\,}
\renewcommand{\Re}{\textup{Re}\,}
\renewcommand{\Im}{\textup{Im}\,}

\numberwithin{equation}{section}

\newtheorem{theorem}{Theorem}
\newtheorem{lemma}{Lemma}[section]
\newtheorem{proposition}{Proposition}[section]
\newtheorem{corollary}{Corollary}[section]

\numberwithin{theorem}{section}
\numberwithin{proposition}{section}
\numberwithin{corollary}{section}

\theoremstyle{remark}

\theoremstyle{definition}
\newtheorem{definition}{Definition}[section]

\title{Least Non-Zero Singular Value and the Distribution of Eigenvectors of non-Hermitian Random Matrices}
\author{Mohammed Osman\footnote{mohammed.osman@qmul.ac.uk}\\~\\\small Queen Mary, University of London}
\date{}

\begin{document}
\maketitle

\abstract{We obtain a tail bound for the least non-zero singular value of $A-z$ when $A$ is a random matrix and $z$ is an eigenvalue of $A$ in a neighbourhood of a given point $z_{0}$ in the bulk of the spectrum. The argument relies on a resolvent comparison and a tail bound for Gauss-divisible matrices. The latter can be obtained by the method of partial Schur decomposition. Using this bound we prove that any finite collection of components of a right eigenvector corresponding to an eigenvalue uniformly sampled from a neighbourhood of a point in the bulk is Gaussian. A byproduct of the calculation is an asymptotic formula for the odd moments of the absolute value of the characteristic polynomial of real Gauss-divisible matrices.}

\section{Introduction}
We are concerned with obtaining tail bounds for the smallest non-zero singular value of $A-z_{n}$ when $z_{n}$ is an eigenvalue of $A$. Our motivation comes from the study of eigenvector statistics of non-Hermitian matrices. Following the work of Knowles and Yin \cite{knowles_eigenvector_2013} in the Hermitian case, the first step towards a comparison theorem between eigenvectors of random matrices satisfying a moment matching condition is to approximate the entries of the eigenvectors by a function of the resolvent. For Hermitian Wigner matrices we expect that the components of an eigenvector $\mbf{u}_{n}$ corresponding to the eigenvalue $\lambda_{n}$ satisfy
\begin{align*}
    u_{n}(j)\bar{u}_{n}(k)&\sim\int_{\lambda_{n}-N^{\xi}\eta}^{\lambda_{n}+N^{\xi}\eta}\Im G_{jk}(E+i\eta)dE
\end{align*}
for small $\eta\ll N$ with sufficiently high probability. This follows by the spectral decomposition of $G$ and eigenvalue repulsion. Once such an approximation has been established, one can use moment matching and a resolvent comparison to estimate the difference between the expectation values of functions of eigenvectors of two random matrices.

To adapt this argument to non-Hermitian matrices $A$ we would like an approximation for eigenvectors in terms of the resolvent $G_{z}(w)$ of the Hermitisation of $A-z$. Such a relationship is only possible when $z=z_{n}$ is an eigenvalue of $A$, in which case the spectral decomposition of $G_{z_{n}}$ reads
\begin{align*}
    G_{z_{n}}(w)&=-\frac{1}{w}\begin{pmatrix}\mbf{l}_{n}\mbf{l}_{n}^{*}&0\\0&\mbf{r}_{n}\mbf{r}_{n}^{*}\end{pmatrix}+\sum_{m=1}^{N-1}\frac{1}{s_{m}^{2}-w^{2}}\begin{pmatrix}w\mbf{u}_{m}\mbf{u}_{m}^{*}&s_{m}\mbf{u}_{m}\mbf{v}_{m}^{*}\\s_{m}\mbf{v}_{m}\mbf{u}_{m}^{*}&w\mbf{v}_{m}\mbf{v}_{m}^{*}\end{pmatrix}.
\end{align*}
The second term is a sum over the non-zero singular values $s_{1},...,s_{m-1}$ of $A-z_{n}$ with corresponding left and right singular vectors $\mbf{u}_{m},\mbf{v}_{m}$. To neglect this term after taking the imaginary part and integrating over $\Re w$ we need a tail bound on the smallest non-zero singular value $s_{N-1}(z_{n})$.

When $A$ has i.i.d. entries and $z$ is a fixed complex number, there has been intensive research into obtaining tail bounds for the least singular value. Below we point out a few relevant results; for a more extensive overview one can consult \cite{rudelson_non-asymptotic_2010,tikhomirov_quantitative_2022} and references therein. The bound
\begin{align}
    P\left(s_{N}(z)\leq \eta\right)&\leq CN\eta,
\end{align}
for some $C>0$ was proved in the case of real Gaussian matrices by Sankar--Spielman--Teng \cite{sankar_smoothed_2006}. For general distributions with bounded second moment, Tao--Vu \cite[Theorem 3.2]{tao_smooth_2010} obtained a slightly weaker version. Nguyen \cite{nguyen_random_2018} obtained overcrowding estimates of the form
\begin{align}
    P\left(s_{N-k+1}(z)\leq\eta\right)&\leq CN^{k-1}(kp(\eta))^{(k-1)^{2}},
\end{align}
where $p(\eta)=\sup_{x}P(|a_{ij}-x|<N\eta)$. For matrices whose entries have bounded density, an improved bound has been obtained by Erd\H{o}s--Ji \cite{erdos_wegner_2023}:
\begin{align}
    P\left(s_{N-k+1}(z)\leq\eta\right)&\leq\begin{cases} CN^{\delta}(N\eta)^{2k^{2}}&\quad \text{complex entries}\\
    CN^{\delta}(N\eta)^{k^{2}}&\quad \text{real entries}
\end{cases},
\end{align}
for any $\delta>0$, which is optimal up to the factor $N^{\delta}$.

After establishing a comparison theorem, the problem of eigenvector distributions is reduced to the study of Gauss-divisible matrices. In the Hermitian case, following Bourgade--Yau \cite{bourgade_eigenvector_2017}, this can be done by analysing the eigenvector flow induced by Dyson Brownian motion. In the non-Hermitian case, the analogous flow is much more complicated and as yet no progress has been made towards proving universality directly from this flow. There is however an alternative approach based on explicit formulas obtained by a partial Schur decomposition, which was used to prove bulk universality of the eigenvalue correlation functions in \cite{maltsev_bulk_2023,osman_bulk_2024,dubova_bulk_2024}. As we will see below, this method can be extended to obtain results about eigenvector distributions as well. The drawback of this approach is that we cannot handle individual eigenvalues/eigenvectors but instead need to take a sum over small neighbourhoods centred at points in the support of the global eigenvalue density (this drawback is also a feature of the use of Girko's formula).

A few days before the first version of this manuscript was uploaded to the arXiv the work of Dubova--Yang--Yau--Yin \cite{dubova_gaussian_2024} appeared. They consider complex matrices and prove Gaussian statistics for eigenvectors associated to several eigenvalues separated by mesoscopic distances $N^{-1/2+\epsilon}$, i.e. the expectation value in \eqref{eq:multiple} below for $|w_{j}-w_{k}|>N^{-1/2+\epsilon},\,j\neq k$. Their approach consists of approximating the moment generating function on a high probability event, whereas we approximate the moments themselves. For the resolvent comparison they also obtain a version of \Cref{thm1} below (Lemma 24 in their paper) by a different argument relying on the universality of the two smallest singular values of $A-z$.

\section{Main Results}
\paragraph{Notation} $\mbb{M}_{n}(\mbb{F}),\,\mbb{M}^{h}_{n}(\mbb{F}),\,\mbb{M}^{sym}_{n}(\mbb{F}),\,\mbb{M}^{skew}_{n}(\mbb{F})$ denote respectively the spaces of general, Hermitian, symmetric and skew-symmetric matrices with entries in $\mbb{F}$. For $M\in\mbb{M}_{n}(\mbb{F})$, $|M|=\sqrt{M^{*}M}$, $\|M\|$ denotes the operator norm and $\|M\|_{2}$ the Frobenius norm. The real and imaginary parts of $M$ are defined by $\Re M=\frac{1}{2}(M+M^{*})$ and $\Im M=\frac{1}{2i}(M-M^{*})$ respectively. $U(n),O(n)$ and $USp(n)$ denote the unitary, orthogonal and unitary symplectic groups respectively. We define the matrices
\begin{align}
    E_{1,n}&=\begin{pmatrix}1&0\\0&0\end{pmatrix}\otimes1_{n},\quad E_{2,n}=\begin{pmatrix}0&0\\1&0\end{pmatrix}\otimes1_{n},\quad E_{3,n}=\begin{pmatrix}0&0\\0&1\end{pmatrix}\otimes1_{n}.\label{eq:Pauli}
\end{align}
We denote by $\mbb{C}_{+}$ the open upper half-plane and $\mbb{D}\subset\mbb{C}$ the open unit disk. When $x$ belongs to a coset space of a compact Lie group (e.g. $U(n),\,O(n)/O(n-m)$), we denote by $d_{H}x$ the Haar measure.

\begin{definition}\label{def1}
We say that $A=(a_{jk})_{j,k=1}^{N}$ is a \textbf{non-Hermitian Wigner matrix} if $a_{jk}$ are independent complex random variables such that $\Re a_{jk}$ is independent of $\Im a_{jk}$ and
\begin{align}
    \mbb{E}\left[a_{jk}\right]&=0,\label{cond1}\\
    \mbb{E}\left[N|a_{jk}|^{2}\right]&=1,\label{cond2}\\
    \mbb{E}\left[N^{p/2}(\Re a_{jk})^{p-q}(\Im a_{jk})^{q}\right]&\leq C_{p},\quad p>2,\,0\leq q\leq p.\label{cond3}
\end{align}
\end{definition}
This includes the case when $\xi$ is real, i.e. $\Im \xi$ is identically zero.

\begin{definition}\label{def2}
We say that $A$ and $B$ are $t$-\textbf{matching} for some $t\geq0$ if they are independent non-Hermitian Wigner matrices such that
\begin{align}
    \mbb{E}\left[\Re(a_{jk})^{p-q}\Im(a_{jk})^{q}\right]&=\mbb{E}\left[\Re(b_{jk})^{p-q}\Im(b_{jk})^{q}\right],\quad p=1,2,3,\,q=0,...,p,
\end{align}
and
\begin{align}
    \left|\mbb{E}\left[\Re(a_{jk})^{4-q}\Im(a_{jk})^{q}\right]-\mbb{E}\left[\Re(b_{jk})^{4-q}\Im(b_{jk})^{q}\right]\right|&\leq\frac{t}{N^{2}},\quad q=0,...,4.
\end{align}
\end{definition}

We denote by $s_{1}(z)\geq\cdots\geq s_{N}(z)$ the singular values of $A_{z}:=A-z$. The corresponding left and right singular vectors are denoted by $\mbf{u}_{n}(z)$ and $\mbf{v}_{n}(z)$ respectively, i.e. $A_{z}\mbf{v}_{n}(z)=s_{n}(z)\mbf{u}_{n}(z)$. We will often suppress the $z$ dependence of $s_{n},\mbf{u}_{n},\mbf{v}_{n}$. $A$ has a bi-orthogonal basis of left and right eigenvectors $\mbf{l}_{n},\mbf{r}_{n}$, which we normalise such that 
\begin{align}
    \|\mbf{r}_{n}\|_{2}&=1,\quad\mbf{l}_{n}^{*}\mbf{r}_{m}=\delta_{nm}.\label{eq:normalisation}
\end{align}
The Hermitisation of $A_{z}$ is the matrix
\begin{align}
    \mc{H}_{z}&=\begin{pmatrix}0&A_{z}\\A^{*}_{z}&0\end{pmatrix},
\end{align}
with resolvent $G_{z}(w)=(\mc{H}_{z}-w)^{-1}$. We also define the resolvents
\begin{align}
    H_{z}(\eta)&=(\eta^{2}+|A_{z}|^{2})^{-1},\\
    \wt{H}_{z}(\eta)&=(\eta^{2}+|A^{*}_{z}|^{2})^{-1},
\end{align}
which appear in the block decomposition of $G_{z}(i\eta)$:
\begin{align}
    G_{z}(i\eta)&=\begin{pmatrix}i\eta \wt{H}_{z}(\eta)&X_{z}H_{z}(\eta)\\H_{z}(\eta)X^{*}_{z}&i\eta H_{z}(\eta)\end{pmatrix}.
\end{align}
In the rest of the paper $N$ is a large integer tending to infinity and $\Tr{M}=N^{-1}\tr M$ is the trace normalised by $N$ (regardless of the dimension of $M$).

Our first result is a bound on the least non-zero singular value of $A-z_{n}$ (equivalently, the second least singular value $s_{N-1}(z_{n})$) for bulk eigenvalues $z_{n}$. We restrict our attention to the bulk in order to reuse some of the analysis from \cite{maltsev_bulk_2023,osman_bulk_2024}.
\begin{theorem}\label{thm1}
Let $A$ be a non-Hermitian Wigner matrix. Let $z_{1},...,z_{N}$ denote the eigenvalues of $A$ and $s_{N}(z)\leq\cdots\leq s_{1}(z)$ denote the singular values of $A-z$. Then for any fixed $\epsilon\in(0,1/84)$, $\xi>0$, $r>0$ and $z\in\mbb{D}$ there is a constant $C_{\epsilon,r,z}$ such that
\begin{align}
    P\left(\min_{\sqrt{N}|z_{n}-z|<r}s_{N-1}(z_{n})<N^{-1-\epsilon}\right)&\leq C_{\epsilon,r,z}\left(N^{-\epsilon}+N^{-1/3+\xi+36\epsilon}\right).
\end{align}
\end{theorem}
In the Gauss-divisble case a stronger version of the above bound holds. One can show that with $A=X+\sqrt{t}Y$ and $Y$ a complex Ginibre matrix, then for any (fixed) $\xi,D>0$ and $t\geq N^{-1+\xi}$ we have
\begin{align}
    P\left(\min_{\sqrt{N}|z_{n}-z|<r}s_{N-1}(z_{n})<\eta\right)&\leq C_{r,z}\left(N^{2}\eta^{2}+N^{-D}\right),
\end{align}
uniformly in $\eta\in(0,N^{-1}]$, with analogous statements in the real case. A more refined resolvent comparison might yield improvements to the factor $N^{-1/3+36\epsilon}$ (for example in the complex case or  for real eigenvalues of real matrices one can obtain $N^{-1/2+42\epsilon}$ using the improved local law from \cite{cipolloni_optimal_2023}) but extending the stronger bound to general matrices by this method seems difficult.

Using this result we can prove a comparison theorem for eigenvector distributions. Let $l\in\mbb{N}$, $\mbf{q}_{1},...,\mbf{q}_{l}\in\mbb{C}^{N}$ and $\theta:\mbb{C}\times\mbb{R}^{l}_{+}\to\mbb{C}$ be a measurable function supported in $B_{r}(0)\times\mbb{R}^{l}_{+}$ for some fixed $r>0$. We define the (symmetrised) joint distribution $\rho_{z_{0},\mbf{q}_{1},...,\mbf{q}_{l}}$ of an eigenvalue $z_{n}$ and the components of the corresponding right eigenvector $\mbf{r}_{n}$ along $\mbf{q}_{1},...,\mbf{q}_{l}$ by
\begin{align}
    \int_{\mbb{C}\times\mbb{R}^{l}_{+}}\theta(z,\mbf{x})\rho_{z_{0},\mbf{q}_{1},...,\mbf{q}_{l}}(z,\mbf{x})dzd\mbf{x}&:=\mbb{E}\left[\mc{L}_{\theta}(z_{0},\mbf{q}_{1},...,\mbf{q}_{l})\right],
\end{align}
where
\begin{align}
    \mc{L}_{\theta}(z_{0},\mbf{q}_{1},...,\mbf{q}_{l})&:=\sum_{n=1}^{N}\theta\left(\sqrt{N}(z_{n}-z_{0}),N|\mbf{q}^{*}_{1}\mbf{r}_{n}|^{2},...,N|\mbf{q}^{*}_{l}\mbf{r}_{n}|^{2}\right).
\end{align}
For real matrices we define two separate distributions for real and complex eigenvalues through the statistics
\begin{align}
    \mc{L}^{\mbb{R}}_{\theta}(u_{0},\mbf{q}_{1},...,\mbf{q}_{l})&=\sum_{n=1}^{N_{\mbb{R}}}\theta\left(\sqrt{N}(u_{n}-u_{0}),N|\mbf{q}_{1}^{T}\mbf{r}^{\mbb{R}}_{n}|^{2},...,N|\mbf{q}_{l}^{T}\mbf{r}^{\mbb{R}}_{n}|^{2}\right),\\
    \mc{L}^{\mbb{C}}_{\theta}(z_{0},\mbf{q}_{1},...,\mbf{q}_{l})&=\sum_{n=1}^{N_{\mbb{C}}}\theta\left(\sqrt{N}(z_{n}-z_{0}),N|\mbf{q}_{1}^{*}\mbf{r}^{\mbb{C}}_{n}|^{2},...,N|\mbf{q}_{l}^{*}\mbf{r}^{\mbb{C}}_{n}|^{2}\right),
\end{align}
where the sums are over the $N_{\mbb{R}}$ real eigenvalues and $N_{\mbb{C}}$ complex eigenvalues in the upper half-plane respectively.
\begin{theorem}\label{thm2}
Let $\epsilon>0$ be fixed and $t=N^{-\epsilon}$. Fix $r>0$ and let $\theta:B_{r}(0)\times\mbb{R}^{l}\to\mbb{C}$ be differentiable to fifth order and satisfy
\begin{align}
    \left|\frac{\partial^{m}}{\partial\mbf{x}^{m}}\theta(z,x_{1},...,x_{l})\right|&\leq C(1+\|\mbf{x}\|)^{C},\quad m=(m_{1},...,m_{l}),\,m_{1}+\cdots+m_{l}\leq 5,
\end{align}
for some $C>0$. Fix $z_{0}\in\mbb{D}$ and $u_{0}\in(-1,1)$. Then for any fixed $l\in\mbb{N}$ there is a $\delta>0$ such that
\begin{enumerate}[i)]
\item
if $A$ and $B$ are $t$-matching complex matrices then
\begin{align}
    \left|\mbb{E}_{A}\left[\mc{L}_{\theta}(z_{0},\mbf{q}_{1},...,\mbf{q}_{l})\right]-\mbb{E}_{B}\left[\mc{L}_{\theta}(z_{0},\mbf{q}_{1},...,\mbf{q}_{l})\right]\right|&\leq N^{-\delta};
\end{align}
\item
if $A$ and $B$ are $t$-matching real matrices then
\begin{align}
    \left|\mbb{E}_{A}\left[\mc{L}^{\mbb{R}}_{\theta}(u_{0},\mbf{q}_{1},...,\mbf{q}_{l})\right]-\mbb{E}_{B}\left[\mc{L}^{\mbb{R}}_{\theta}(u_{0},\mbf{q}_{1},...,\mbf{q}_{l})\right]\right|&\leq N^{-\delta}
\end{align}
and
\begin{align}
    \left|\mbb{E}_{A}\left[\mc{L}^{\mbb{C}}_{\theta}(z_{0},\mbf{q}_{1},...,\mbf{q}_{l})\right]-\mbb{E}_{B}\left[\mc{L}^{\mbb{C}}_{\theta}(z_{0},\mbf{q}_{1},...,\mbf{q}_{l})\right]\right|&\leq N^{-\delta}.
\end{align}
\end{enumerate}
\end{theorem}
For simplicity we have restricted our attention to a single eigenvalue/eigenvector pair but the idea of the proof can be applied to expectations of the form
\begin{align}
    \mbb{E}\left[\sum_{j_{1}\neq\cdots\neq j_{m}}\theta\left(\sqrt{N}(z_{j_{1}}-w_{1}),...,\sqrt{N}(z_{j_{m}}-w_{m}),N|\mbf{q}_{1}^{*}\mbf{r}_{j_{1}}|^{2},...,N|\mbf{q}_{m}^{*}\mbf{r}_{j_{m}}|^{2}\right)\right],\label{eq:multiple}
\end{align}
for finite $m$ (in fact one can take $m=N^{\delta}$ for sufficiently small $\delta>0$).

\Cref{thm2} reduces the calculation of the joint distribution to the Gauss-divisible case, which can be done by the method of partial Schur decomposition in a similar way to the proof of bulk universality of the correlation functions. We obtain the following result in the spirit of \cite[Corollary 1.3]{bourgade_eigenvector_2017}.
\begin{theorem}\label{thm3}
Let $\theta(z,\mbf{x})$ be supported in $B_{r}(0)\times\mbb{R}^{l}_{+}$ for some fixed $r>0$ and be polynomial in the entries of $\mbf{x}\in\mbb{R}^{l}$. Then there is a $\delta>0$ such that:
\begin{enumerate}[i)]
\item
for a complex non-Hermitian Wigner matrix $A$ and fixed $z_{0}\in\mbb{D}$ we have
\begin{align}
    \mbb{E}_{A}\left[\mc{L}_{\theta}(z_{0},\mbf{q}_{1},...,\mbf{q}_{l})\right]&=\frac{1}{\pi}\int_{\mbb{C}\times\mbb{R}^{l}_{+}}\theta\left(z,\mbf{x}\right)\rho_{\mbf{q}_{1},...,\mbf{q}_{l}}(\mbf{x})dzd\mbf{x}+O(N^{-\delta}),
\end{align}
where $\rho_{\mbf{q}_{1},...,\mbf{q}_{l}}$ is the density of $(|\mbf{q}_{1}^{*}\mbf{p}|^{2},...,|\mbf{q}_{l}^{*}\mbf{p}|^{2})$ for a standard complex Gaussian vector $\mbf{p}\in\mbb{C}^{N}$;
\item
for a real non-Hermitian Wigner matrix $A$ and fixed $u_{0}\in(-1,1)$ we have
\begin{align}
    \mbb{E}_{A}\left[\mc{L}^{\mbb{R}}_{\theta}(u_{0},\mbf{q}_{1},...,\mbf{q}_{l})\right]&=\frac{1}{\sqrt{2\pi}}\int_{\mbb{R}\times\mbb{R}^{l}_{+}}\theta(u,\mbf{x})\rho_{\mbf{q}_{1},...,\mbf{q}_{l}}(\mbf{x})dud\mbf{x}+O(N^{-\delta}),
\end{align}
where $\rho_{\mbf{q}_{1},...,\mbf{q}_{l}}$ is the density of $(|\mbf{q}_{1}^{T}\mbf{p}|^{2},...,|\mbf{q}_{l}^{T}\mbf{p}|^{2})$ for a standard real Gaussian vector $\mbf{p}\in\mbb{R}^{N}$;
\item
for a real non-Hermitian Wigner matrix $A$ and fixed $z_{0}\in\mbb{C}_{+}$ we have
\begin{align}
    \mbb{E}_{A}\left[\mc{L}^{\mbb{C}}_{\theta}(z_{0},\mbf{q}_{1},...,\mbf{q}_{l})\right]&=\frac{1}{\pi}\int_{\mbb{C}_{+}\times\mbb{R}^{l}_{+}}\theta(z,\mbf{x})\rho_{\mbf{q}_{1},...,\mbf{q}_{l}}(\mbf{x})dzd\mbf{x}+O(N^{-\delta}),
\end{align}
where $\rho_{\mbf{q}_{1},...,\mbf{q}_{l}}$ is the density of $(|\mbf{q}_{1}^{*}\mbf{p}|^{2},...,|\mbf{q}_{l}^{*}\mbf{p}|^{2})$ for a standard complex Gaussian vector $\mbf{p}\in\mbb{C}^{N}$;
\item
for a real non-Hermitian Wigner matrix $A$ and fixed $u_{0}\in(-1,1)$ we have
\begin{align}
    \mbb{E}_{A}\left[\mc{L}^{\mbb{C}}_{\theta}(u_{0},\mbf{q}_{1},...,\mbf{q}_{l})\right]&=\frac{1}{\pi}\int_{\mbb{C}_{+}}\int_{0}^{\infty}\frac{2y\delta}{\sqrt{\delta^{2}+4y^{2}}}e^{-\frac{1}{2}\delta^{2}}\nonumber\\&\times\int_{\mbb{R}^{l}_{+}}\theta(z,\mbf{x})\rho_{\mbf{q}_{1},...,\mbf{q}_{l}}(\mbf{x};\delta,y)d\mbf{x}d\delta dz+O(N^{-\delta}),
\end{align}
where $\rho_{\mbf{q}_{1},...,\mbf{q}_{l}}$ is the density of $(|\mbf{q}_{1}^{*}\mbf{p}|^{2},...,|\mbf{q}_{l}^{*}\mbf{p}|^{2})$ for
\begin{align}
    \mbf{p}&=\sqrt{1+\frac{\delta}{\sqrt{\delta^{2}+4y^{2}}}}\mbf{v}_{1}+i\sqrt{1-\frac{\delta}{\sqrt{\delta^{2}+4y^{2}}}}\mbf{v}_{2},
\end{align}
with independent standard Gaussian vectors $\mbf{v}_{1},\mbf{v}_{2}\in\mbb{R}^{N}$.
\end{enumerate}
\end{theorem}

In other words, if we uniformly sample an eigenvalue from an $O(N^{-1/2})$ neighbourhood of a point $z_{0}\in\mbb{D}$, then any finite collection of components of the corresponding right eigenvector are Gaussian. This result does not cover the distribution of an individual eigenvector, since we need to take a sum in order to apply Girko's formula. It is an open question whether one can prove comparison theorems without recourse to Girko's formula, say by directly estimating the derivatives of eigenvalues and eigenvectors with respect to the matrix elements. The problem here is that these derivatives are much larger than their counterparts in the Hermitian case due to the fact that the product of the $l^{2}$ norms of left and right eigenvectors is typically large.

In proving the above results, we have to study the expectation value $\mbb{E}\left[|\det A_{u}|^{m}\right]$ for real, Gauss-divisible $A$ and $m=1$. Since the case of general $m\in\mbb{N}$ is not much more difficult we include it in the following.
\begin{theorem}\label{thm4}
Let $\epsilon>0$ be fixed, $N^{-1+\epsilon}\leq t\leq 1$ and fix $u\in(-1,1)$. Let $X$ be a real $N\times N$ matrix such that for any fixed $\delta>0$ there are constants $c_{\delta},C_{\delta}$ for which 
\begin{align}
    c_{\delta}\leq \eta\Tr{H_{u}(\eta)}& \leq C_{\delta},\\
    \eta^{3}\Tr{H^{2}_{u}(\eta)}&\geq c_{\delta},
\end{align}
for all $\eta\in[\delta t,t/\delta]$. Define
\begin{align}
    \phi_{u}&=\frac{\eta_{u}^{2}}{t}-\Tr{\log(\eta_{u}^{2}+|X_{u}|^{2})},
\end{align}
where $\eta_{u}$ is such that $t\Tr{H_{u}(\eta_{u})}=1$. Then for any fixed $m\in\mbb{N}$ we have
\begin{align}
    \mbb{E}_{Y}\left[|\det(X_{u}+\sqrt{t}Y)|^{m}\right]&=\left[1+O\left(\frac{\log^{3}N}{\sqrt{Nt}}\right)\right]d_{N,m}(u)e^{-\frac{Nm}{2}\phi_{u}},
\end{align}
where the expectation is with respect to $Y\sim GinOE(N)$,
\begin{align}
    d_{N,m}(u)&=\frac{(2\pi)^{m/2}G(1/2)}{G((m+1)/2)G(m/2+1)}\left(\frac{N}{2t^{2}\Tr{H^{2}_{u}(\eta_{u})}}\right)^{m(m-1)/4},
\end{align}
and $G(z)$ is the Barnes G function.
\end{theorem}
It follows from the local law that non-Hermitian Wigner matrices satisfy the conditions of the above theorem with probability $1-N^{-D}$ for any $D>0$. We can also take $t=1$ and $X=0$ to obtain the asymptotics of the GinOE; in this case we have $\eta_{u}^{2}=1-u^{2}$ and $t^{2}\Tr{H^{2}_{u}}=1$. This gives us the odd integer case of \cite[Conjecture 5.9]{serebryakov_schur_2023}:
\begin{align}
    \mbb{E}_{Y}\left[|\det(Y-u)|^{m}\right]&=\left[1+O\left(\frac{\log^{3}N}{\sqrt{N}}\right)\right]\frac{(2\pi)^{m/2}G(1/2)}{G((m+1)/2)G(m/2+1)}\nonumber\\
    &\times\left(\frac{N}{2}\right)^{m(m-1)/4}e^{-\frac{Nm}{2}(1-u^{2})},\quad m\in\mbb{N}.
\end{align}
The conjecture is for any $m$ in a compact subset of $(-1,\infty)$, but non-integer values are beyond the reach of the supersymmetry method which we employ here.

The rest of the paper is organised as follows. In \Cref{sec3} we collect some existing results. In \Cref{sec4} we prove \Cref{thm4}. In \Cref{sec5} we prove \Cref{thm1}, up to the proof of \Cref{lem:GaussDivisible} which we defer to \Cref{sec6}. In \Cref{sec7} we prove \Cref{thm2} and in \Cref{sec8} we prove \Cref{thm3}.

\section{Preliminaries}\label{sec3}
We record here some existing results which we will need for our arguments. Throughout this section $A$ is a non-Hermitian Wigner matrix. We also make use of the concept of stochastic domination: $X\prec Y$ if for any $\xi,D>0$ we have $|X|\leq N^{\xi}|Y|$ for sufficiently large $N>N(\xi,D)$.

First, we need a bound on the operator norm, which follows from the local law for sample covariance matrices in \cite[Theorem 3.7]{knowles_anisotropic_2017}.
\begin{proposition}
There is a constant $C$ such that
\begin{align}
    \|A\|\prec C.
\end{align}
\end{proposition}

We have already mentioned the following tail bound for the least singular value, which follows from \cite[Theorem 3.2]{tao_smooth_2010}.
\begin{proposition}[Theorem 3.2 in \cite{tao_smooth_2010}]\label{prop:leastSV}
For any fixed $z\in\mbb{D}$, $A>0$ and $\xi>0$ there is a constant $C$ such that
\begin{align}
    P(s_{N}(z)<N^{-A-1})&\leq CN^{\xi-A},
\end{align}
for sufficiently large $N>N(z,A,\xi)$.
\end{proposition}

Next we need one- and two-resolvent local laws for the Hermitisation of $A-z$. The deterministic approximation to a single resolvent is given by
\begin{align}
    M_{z}(w)&=\begin{pmatrix}m_{z}(w)&-zu_{z}(w)\\-\bar{z}u_{z}(w)&m_{z}(w)\end{pmatrix}
\end{align}
where
\begin{align}
    u_{z}(w)&=\frac{m_{z}(w)}{w+m_{z}(w)}
\end{align}
and $m_{z}(w)$ is the unique solution to
\begin{align}
    -\frac{1}{m_{z}}&=w+m_{z}-\frac{|z|^{2}}{w+m_{z}},\quad \Im w\cdot\Im m_{z}>0.
\end{align}
Let $S:\mbb{M}_{2n}\to\mbb{M}_{2n}$ and $\mc{B}_{z_{1},z_{2}}(z_{1},z_{2}):\mbb{M}_{2n}\to\mbb{M}_{2n}$ be defined by
\begin{align}
    S\left[\begin{pmatrix}A&B\\C&D\end{pmatrix}\right]&=\begin{pmatrix}\Tr{D}&0\\0&\Tr{A}\end{pmatrix},\\
    \mc{B}_{z_{1},z_{2}}(w_{1},w_{2})[F]&=1-M_{z_{1}}(w_{1})S\left[F\right]M_{z_{2}}(w_{2}).
\end{align}
The deterministic approximation to $G_{z_{1}}(w_{1})FG_{z_{2}}(w_{2})$ is given by 
\begin{align}
    M_{z_{1},z_{2}}(w_{1},w_{2},F)&=(\mc{B}_{z_{1},z_{2}}(w_{1},w_{2}))^{-1}\left[M_{z_{1}}(w_{1})FM_{z_{2}}(w_{2})\right].
\end{align}
\begin{proposition}[Theorem 2.6 in \cite{cipolloni_optimal_2023}, Theorem 5.2 in \cite{cipolloni_central_2023} and Theorem 3.3 in \cite{cipolloni_mesoscopic_2023}]\label{prop:ll1}
Let $\epsilon>0$, $z_{1},z_{2}\in\mbb{C}$, $\mbf{x},\mbf{y}\in S^{N-1}$, $F_{j}\in\mbb{M}_{N}$, $w\in\mbb{C}$ such that $|\Re w|<1$ and $\eta,\sigma\in\mbb{R}$ such that $\eta_{*}=\min(|\eta|,|\sigma|,|\Im w|)>N^{-1+2\epsilon}$. Then
\begin{align}
    \left|\Tr{\left[G_{z_{1}}(w)-M_{z_{1}}(w)\right]F_{1}}\right|&\prec\frac{1}{N\eta_{*}},\\
    \left|\langle\mbf{x},\left[G_{z_{1}}(w)-M_{z_{1}}(w)\right]\mbf{y}\rangle\right|&\prec\frac{1}{\sqrt{N\eta_{*}}},
\end{align}
and
\begin{align}
    \left|\Tr{\left[G_{z_{1}}(i\eta)F_{1}G_{z_{2}}(i\sigma)-M_{z_{1},z_{2}}(i\eta,i\sigma,F_{1})\right]F_{2}}\right|&\prec\frac{N^{\epsilon}}{N\eta_{*}^{2}}.
\end{align}
\end{proposition}
We have combined the statement of \cite[Theorem 5.2]{cipolloni_central_2023}, which is valid for any $z_{1},z_{2}\in\mbb{C}$ but contains an extra factor of $\|\mc{B}_{z_{1},z_{2}}^{-1}\|\sim(|z_{1}-z_{2}|^{2}+\eta+\sigma)^{-1}$ in the error term, with the statement of \cite[Theorem 3.3]{cipolloni_mesoscopic_2023}, which is valid for $|z_{1}-z_{2}|<N^{-\epsilon}$ and does not contain this extra factor. Note also that in \cite[Theorem 5.2]{cipolloni_central_2023} and \cite[Theorem 3.3]{cipolloni_mesoscopic_2023} the matrix $A$ is assumed to have identically distributed entries, but it is enough that the entries have the same variance $1/N$.

We will also make use of the improved local law when taking products of resolvents with matrices whose diagonal blocks are zero.
\begin{proposition}[\cite{cipolloni_optimal_2023} Theorems 4.3 and 4.4]\label{prop:ll2}
Let $A$ be a non-Hermitian Wigner matrix and $z\in\mbb{D}$. Let $\epsilon>0$, $F_{j}\in\{E_{2,N},E^{*}_{2,N}\}$, $\mbf{x},\mbf{y}\in S^{N-1}$, $w_{j}\in\mbb{C}$ such that $|\Re w_{j}|<1$ and $\eta=\min_{j}|\Im w_{j}|>N^{-1+\epsilon}$. Then
\begin{align}
    \left|\Tr{\left[G_{z}(w_{1})-M_{z}(w_{1})\right]F_{1}}\right|&\prec\frac{1}{N\eta^{1/2}},\\
    \left|\langle\mbf{x},G_{z}(w_{1})F_{1}G_{z}(w_{2})\mbf{y}\rangle\right|&\prec 1+\frac{1}{\sqrt{N\eta^{2}}},
\end{align}
and
\begin{align}
    \left|\Tr{\left[G_{z}(w_{1})F_{1}G_{z}(w_{2})-M_{z,z}(w_{1},w_{2},F_{1})\right]F_{2}}\right|&\prec\frac{1}{\sqrt{N\eta}},\\
    \left|\langle\mbf{x},G_{z}(w_{1})F_{1}G_{z}(w_{2})F_{2}G_{z}(w_{3})\mbf{y}\rangle\right|&\prec\frac{1}{\eta}.
\end{align}
\end{proposition}
The condition $|\Re w|<1$ is made to ensure that $\Re w$ is well inside the bulk of the singular value distribution; we could replace 1 with any small fixed constant.

A standard consequence of the isotropic local law is delocalisation for singular vectors. As observed in \cite[Theorem 2.7]{cipolloni_optimal_2023}, the improved local law also implies a bound on the overlaps between singular vectors. We collect these statements in the following proposition.
\begin{proposition}[Theorem 2.7 in \cite{cipolloni_optimal_2023}]\label{prop:deloc}
Let $\mbf{q}\in S^{N-1}$. We have
\begin{align}
    \max_{s_{n}(z)<1}|\mbf{q}^{*}\mbf{u}_{n}|&\prec N^{-1/2},\label{eq:deloc1}\\
    \max_{s_{n}(z)<1}|\mbf{q}^{*}\mbf{v}_{n}|&\prec N^{-1/2},\label{eq:deloc2}\\
    \max_{s_{n}(z),s_{m}(z)<1}|\mbf{u}_{n}^{*}\mbf{v}_{m}|&\prec N^{-1/2}\label{eq:singularOverlap}.
\end{align}
\end{proposition}

Using this, we can obtain bounds for certain traces of resolvents when the spectral parameter is arbitrarily close to the real axis, which we will need when analysing Girko's formula.
\begin{lemma}\label{lem:aPriori}
Let $\eta\in(0,N^{-1}]$. Then we have
\begin{align}
    \left|\Tr{G_{z}(i\eta)}\right|&\prec \frac{1}{N\eta},\label{eq:ext1}\\
    \left|\Tr{G^{2}_{z}(i\eta)F}\right|&\prec \frac{1}{N^{3/2}\eta^{2}},\label{eq:ext2}\\
    \left|\Tr{G^{2}_{z}(i\eta)FG_{z}(i\eta)F^{*}}\right|&\prec \frac{1}{N^{2}\eta^{3}},\label{eq:ext3}
\end{align}
and for $w=E+i\eta,\,|E|\in[0,N^{-1}]$ and $\mbf{x},\mbf{y}\in S^{2N-1}(\mbb{C})$ we have
\begin{align}
    \left|\mbf{x}^{*}G_{z}(w)\mbf{y}\right|&\prec\frac{|w|}{N\eta^{2}},\label{eq:ext4}\\
    \left|\mbf{x}^{*}G_{z}(w)FG_{z}(w)\mbf{y}\right|&\prec\frac{|w|^{2}}{N^{3/2}\eta^{4}},\label{eq:ext5}\\
    \left|\mbf{x}^{*}G_{z}(w)FG_{z}(w)F^{*}G_{z}(w)\mbf{y}\right|&\prec\frac{|w|^{3}}{N^{2}\eta^{6}}.\label{eq:ext6}
\end{align}
\end{lemma}
\begin{proof}
For $\sigma=N^{-1+\xi}$, the local law implies that
\begin{align*}
    \Im\Tr{G_{z}(i\sigma)}&\prec 1.
\end{align*}
Thus with $\eta<\sigma$ we find
\begin{align*}
    \Im\Tr{G_{z}(i\eta)}&\leq\frac{\sigma}{\eta}\Im\Tr{G_{z}(i\sigma)}\prec\frac{1}{N\eta}.
\end{align*}
This argument to extend the domain of the local law has appeared on several occasions (see e.g. section 10.1 in the notes of Benaych-Georges and Knowles \cite{benaych-georges_lectures_2019}).

For the remaining estimates we make use of the spectral decomposition of $G_{z}(w)$:
\begin{align}
    G_{z}(w)&=\sum_{n=1}^{N}\frac{2}{s^{2}_{n}(z)-w^{2}}\begin{pmatrix}w\mbf{u}_{n}\mbf{u}_{n}^{*}&s_{n}(z)\mbf{u}_{n}\mbf{v}_{n}^{*}\\s_{n}(z)\mbf{v}_{n}\mbf{u}_{n}^{*}&w\mbf{v}_{n}\mbf{v}_{n}^{*}\end{pmatrix},
\end{align}
where $A_{z}\mbf{v}_{n}=s_{n}(z)\mbf{u}_{n}$.

Consider the left-hand side of \eqref{eq:ext2}:
\begin{align*}
    \left|\Tr{G^{2}_{z}(i\eta)F}\right|&=\left|\frac{1}{N}\sum_{n=1}^{N}\frac{8i\eta s_{n}(z)\mbf{u}_{n}^{*}\mbf{v}_{n}}{(s^{2}_{n}(z)+\eta^{2})^{2}}\right|\\
    &\leq\frac{4}{N}\sum_{n=1}^{N}\frac{|\mbf{u}_{n}^{*}\mbf{v}_{n}|}{s^{2}_{n}(z)+\eta^{2}},
\end{align*}
where we have used the inequality $|xy|\leq(x^{2}+y^{2})/2$. We split the sum into two parts consisting of the singular values greater or less than 1. For the former we use the trivial bound $|\mbf{u}_{n}^{*}\mbf{v}_{m}|\leq 1/2$ and for the latter we use \eqref{eq:singularOverlap}:
\begin{align*}
    \left|\Tr{G^{2}_{z}(i\eta)F}\right|&\prec 1+\frac{1}{N^{3/2}}\sum_{s_{n}(z)<1}\frac{1}{s^{2}_{n}(z)+\eta^{2}}\\
    &\prec 1+\frac{1}{N^{1/2}\eta}\Im\Tr{G_{z}(i\eta)}\\
    &\prec\frac{1}{N^{3/2}\eta^{2}}.
\end{align*}
Consider now \eqref{eq:ext3}:
\begin{align*}
    \left|\Tr{G^{2}_{z}(i\eta)FG_{z}(i\eta)F^{*}}\right|&=\left|\frac{1}{N}\sum_{n,m=1}^{N}\frac{4\eta(s^{2}_{n}(z)-\eta^{2})|\mbf{u}_{n}^{*}\mbf{v}_{m}|^{2}}{(s^{2}_{n}(z)+\eta^{2})^{2}(s^{2}_{m}(z)+\eta^{2})}\right|\\
    &\leq \frac{4\eta}{N}\sum_{n,m=1}^{N}\frac{|\mbf{u}_{n}^{*}\mbf{v}_{m}|^{2}}{(s^{2}_{n}(z)+\eta^{2})(s^{2}_{m}(z)+\eta^{2})}.
\end{align*}
We split the sum as before and use \eqref{eq:singularOverlap} to obtain
\begin{align*}
    \left|\Tr{G^{2}_{z}(i\eta)FG_{z}(i\eta)F^{*}}\right|&\prec \frac{1}{\eta}+\frac{1}{\eta}\left(\frac{1}{N}\sum_{s_{n}(z)<1}\frac{2\eta}{s^{2}_{n}(z)+\eta^{2}}\right)^{2}\\
    &\prec\frac{1}{N^{2}\eta^{3}}.
\end{align*}
The proofs of \eqref{eq:ext4}, \eqref{eq:ext5} and \eqref{eq:ext6} are similar, except we also make use of the delocalisation bounds in \eqref{eq:deloc1} and \eqref{eq:deloc2}. Let us demonstrate \eqref{eq:ext6} for $\mbf{x}=(\mbf{e}_{j},0)^{T}$ and $\mbf{y}=(0,\mbf{e}_{k})$, the left hand side of which we denote by $x$. Applying the spectral decomposition we find
\begin{align*}
    x&=-8\sum_{n,m,l=1}^{N}\frac{w^{2}s_{n}(\mbf{e}_{j}^{*}\mbf{u}_{n})(\mbf{u}_{n}^{*}\mbf{v}_{m})(\mbf{v}_{m}^{*}\mbf{u}_{l})(\mbf{v}_{l}^{*}\mbf{e}_{k})}{(s_{n}^{2}-w^{2})(s_{m}^{2}-w^{2})(s_{l}^{2}-w^{2})}.
\end{align*}
We split the sum according to whether $s_{\mu}$ is greater or less than 1 for $\mu=n,m,l$. When $s_{\mu}>1$ we use the bound $|s_{\mu}^{2}-w^{2}|>C$ to remove $s_{\mu}$-dependent terms and then extend the sum to all indices using Cauchy-Schwarz and the fact that $\sum_{n}\mbf{u}_{n}\mbf{u}_{n}^{*}=\sum_{n}\mbf{v}_{n}\mbf{v}_{n}^{*}=1/2$. For example, when $s_{l}>1$ we have
\begin{align*}
    \left|\sum_{s_{l}>1}\frac{w^{2}s_{n}(\mbf{e}_{j}^{*}\mbf{u}_{n})(\mbf{u}_{n}^{*}\mbf{v}_{m})(\mbf{v}_{m}^{*}\mbf{u}_{l})(\mbf{v}_{l}^{*}\mbf{e}_{k})}{(s_{n}^{2}-w^{2})(s_{m}^{2}-w^{2})(s_{l}^{2}-w^{2})}\right|&\leq\frac{C|w|^{2}s_{n}|\mbf{e}_{j}^{*}\mbf{u}_{n}|\cdot|\mbf{u}_{n}^{*}\mbf{v}_{m}|}{|s_{n}^{2}-w^{2}|\cdot|s_{m}^{2}-w^{2}|}\sqrt{\sum_{l}|\mbf{v}_{m}^{*}\mbf{u}_{l}|^{2}}\sqrt{\sum_{l}|\mbf{v}_{l}^{*}\mbf{e}_{k}|^{2}}\\
    &=\frac{C|w|^{2}s_{n}|\mbf{e}_{j}^{*}\mbf{u}_{n}|\cdot|\mbf{u}_{n}^{*}\mbf{v}_{m}|}{|s_{n}^{2}-w^{2}|\cdot|s_{m}^{2}-w^{2}|}.
\end{align*}
When $s_{\mu}<1$ for $\mu=n$ or $\mu=l$ and $s_{m}<1$ we use $|\mbf{u}_{\mu}^{*}\mbf{v}_{m}|\prec N^{-1/2}$ and $|\mbf{e}_{j}^{*}\mbf{u}_{\mu}|\prec N^{-1/2}$. The largest contribution is from the terms with $s_{n},s_{m},s_{l}<1$:
\begin{align*}
    \left|\sum_{s_{n},s_{m},s_{l}<1}\frac{w^{2}s_{n}(\mbf{e}_{j}^{*}\mbf{u}_{n})(\mbf{u}_{n}^{*}\mbf{v}_{m})(\mbf{v}_{m}^{*}\mbf{u}_{l})(\mbf{v}_{l}^{*}\mbf{e}_{k})}{(s_{n}^{2}-w^{2})(s_{m}^{2}-w^{2})(s_{l}^{2}-w^{2})}\right|&\prec\frac{1}{N^{2}}\left(\sum_{n=1}^{N}\frac{s_{n}}{|s_{n}^{2}-w^{2}|}\right)\left(\sum_{n=1}^{N}\frac{|w|}{|s_{n}^{2}-w^{2}|}\right)^{2}.
\end{align*}
To deal with the remaining sums we consider separately the cases $|E|<\eta$ and $|E|>\eta$. In the former case we have
\begin{align*}
    |s_{n}^{2}-w^{2}|\geq\frac{s_{n}^{2}+\eta^{2}}{2},\quad |E|<\eta,
\end{align*}
and so
\begin{align*}
    \sum_{n=1}^{N}\frac{|w|}{|s_{n}^{2}-w^{2}|}&\leq\sum_{n=1}^{N}\frac{2|w|}{s_{n}^{2}+\eta^{2}}\prec\frac{|w|}{\eta^{2}}.
\end{align*}
Let $\sigma=N^{-1+\xi}$; then we have
\begin{align*}
    \sum_{n=1}^{N}\frac{s_{n}}{s_{n}^{2}+\eta^{2}}&\leq\sum_{n=1}^{N}\frac{s_{n}}{s_{n}^{2}+\sigma^{2}}+\sum_{n=1}^{N}\frac{(\sigma^{2}-\eta^{2})s_{n}}{(s_{n}^{2}+\eta^{2})(s_{n}^{2}+\sigma^{2})}\\
    &\leq\sum_{n=1}^{N}\frac{s_{n}}{s_{n}^{2}+\sigma^{2}}+\left(\sum_{n=1}^{N}\frac{s_{n}^{2}}{(s_{n}^{2}+\eta^{2})^{2}}\right)^{1/2}\left(\sum_{n=1}^{N}\frac{(\sigma^{2}-\eta^{2})^{2}}{(s_{n}^{2}+\sigma^{2})^{2}}\right)^{1/2}\\
    &\prec\sum_{n=1}^{N}\frac{s_{n}}{s_{n}^{2}+\sigma^{2}}+\frac{1}{\eta}\left(\sum_{n=1}^{N}\frac{\sigma^{4}}{(s_{n}^{2}+\sigma^{2})^{2}}\right)^{1/2}\\
    &\prec\sum_{n=1}^{N}\frac{s_{n}}{s_{n}^{2}+\sigma^{2}}+\frac{1}{\eta},
\end{align*}
where in the last inequality we have used $\Tr{G_{z}(i\sigma)}\prec1$. Using a dyadic decomposition and the local law we estimate the first term on the last line as follows
\begin{align*}
    \sum_{n=1}^{N}\frac{s_{n}}{s_{n}^{2}+\sigma^{2}}&\leq \frac{|\{n:s_{n}<\sigma\}|}{\sigma}+4\sum_{k=1}^{C\lfloor\log N\rfloor}\sum_{2^{k-1}\sigma<s_{n}<2^{k}\sigma}\frac{|\{n:s_{n}<2^{k}\sigma\}|}{(2^{k}+2^{-k+2})\sigma}\\
    &\prec N\log N\\
    &\prec\frac{1}{\eta}.
\end{align*}
Thus when $|E|<\eta$ we have
\begin{align*}
    \frac{1}{N^{2}}\left(\sum_{n=1}^{N}\frac{s_{n}}{|s_{n}^{2}-w^{2}|}\right)\left(\sum_{n=1}^{N}\frac{|w|}{|s_{n}^{2}-w^{2}|}\right)^{2}&\prec\frac{|w|^{2}}{N^{2}\eta^{4}}\prec\frac{|w|^{3}}{N^{3}\eta^{6}}.\label{eq:|E|<eta}
\end{align*}
When $|E|>\eta$, we split the sum into $s_{n}<10E$ and $s_{n}>10E$. When $s_{n}<10E$, we use the bound $|s_{n}^{2}-w^{2}|>2|E|\eta>\eta^{2}$. Since $|E|<N^{-1}$, the local law implies that $|\{n:s_{n}<10E\}|\prec1$ and so
\begin{align*}
    \sum_{s_{n}<10E}\frac{|w|}{|s_{n}^{2}-w^{2}|}&\prec\frac{|w|}{\eta^{2}},\\
    \sum_{s_{n}<10E}\frac{s_{n}}{|s_{n}^{2}-w^{2}|}&\prec\frac{|w|}{\eta^{2}}.
\end{align*}
When $s_{n}>10E$, we have $|s_{n}^{2}-w^{2}|>(s_{n}^{2}+\eta^{2})/2$ and can repeat the previous steps. Thus we find that \eqref{eq:|E|<eta} also holds when $|E|>\eta$.
\end{proof}

When using the Lindenberg replacement strategy, it is important that these statements (except the bound on the least singular value) also hold if we set the real or imaginary part of one element of $A$ to zero. For later reference, we record this fact in the following corollary.
\begin{corollary}\label{cor}
The statements of \Cref{prop:ll1,prop:ll2,prop:deloc} and \Cref{lem:aPriori} hold if $\Re a_{jk}$ or $\Im a_{jk}$ is set to zero for a single $(j,k)$ and the remaining elements satisfy \eqref{cond1}, \eqref{cond2} and \eqref{cond3}.
\end{corollary}
\begin{proof}
Since \Cref{prop:deloc} and \Cref{lem:aPriori} follow from \Cref{prop:ll1} and \Cref{prop:ll2}, it is enough to show that the latter two continue to hold. If $A^{(jk)}$ denotes the matrix obtained from $A$ by setting $\Re a_{jk}=0$ and $G^{(jk)}_{z}$ the resolvent of the corresponding Hermitisation, then we have
\begin{align}
    G^{(jk)}_{z}(w)&=\sum_{q=0}^{p-1}(\Re a_{jk})^{q}(G_{z}(w)\Delta_{jk})^{q}G_{z}(w)+(\Re a_{jk})^{p}(G_{z}(w)\Delta_{jk})^{p}G^{(jk)}_{z}(w),
\end{align}
for any $p\in\mbb{N}$, where $\Delta_{jk}$ is the matrix with -1 in the $(j,N+k)$ and $(N+j,k)$ entries and zero elsewhere. We insert this into various trace expressions and use the facts that $|a_{jk}|\prec N^{-1/2}$ and $|G_{z,jk}(w)|\prec 1$ for $|\Im w|>N^{-1+\epsilon}$.
\end{proof}

Now we make some preparations for the calculations with Gauss-divisible matrices. For $V\in\mbb{C}^{N\times k}$ such that $V^{*}V=1_{k}$ we denote by $X^{(V)}$ the projection of $X$ onto the orthogonal complement of the span of the columns of $V$. We define the function $\phi_{z}:\mbb{R}_{+}\to\mbb{R}$ by
\begin{align}
    \phi_{z}(\eta)&=\frac{\eta^{2}}{t}-\Tr{\log(\eta^{2}+|X_{z}|^{2})}\label{eq:phi}
\end{align}
which will play a central role in the asymptotic analysis of certain integrals. The minimum occurs at the point $\eta_{z}$ that satisfies 
\begin{align}
    t\Tr{H_{z}(\eta_{z})}&=1.\label{eq:eta_zt}
\end{align}
It can be shown (see \cite[Lemmas 3.5 and 3.6]{osman_bulk_2024}) that on the event that the local laws in \Cref{prop:ll1,prop:ll2} hold we have
\begin{align}
    t/C<\eta_{z}&<Ct,\label{eq:etaBound}\\
    \phi_{z}(\eta)-\phi_{z}(\eta_{z})&\geq\frac{C(\eta-\eta_{z})^{2}}{t},\label{eq:phiBound}
\end{align}
for any $z\in\mbb{D}$ and $N^{-1+\epsilon}\leq t\leq1$. The presence of a superscript $(V)$ means that we replace $X$ with $X^{(V)}$ in the relevant quantity, e.g. 
\begin{align*}
    H^{(V)}_{z}(\eta)&=(|X^{(V)}_{z}|^{2}+\eta^{2})^{-1},\\
    \phi^{(V)}_{z}(\eta)&=\frac{\eta^{2}}{t}-\Tr{\log(\eta^{2}+|X^{(V)}_{z}|^{2})},
\end{align*}
and so on. By interlacing it follows that
\begin{align}
    |\eta^{(V)}_{z}-\eta_{z}|&\leq\frac{C\textup{rank}(V)}{N}.\label{eq:etaBound2}
\end{align}
We suppress the argument of $\phi_{z},H_{z},G_{z}$ when they are evaluated at $\eta_{z}$, i.e. $\phi_{z}=\phi_{z}(\eta_{z})$ etc. We also define the quantities
\begin{align}
    \sigma_{z}&=\eta^{2}_{z}\Tr{H_{z}\wt{H}_{z}}+\frac{|\Tr{H^{2}_{z}X_{z}}|^{2}}{\Tr{H^{2}_{z}}},\label{eq:sigma}\\
    \wt{\sigma}_{z}&=\eta_{z}^{2}\Tr{H_{\bar{z}}\wt{H}_{z}}+\frac{|\Tr{H_{\bar{z}}X_{z}H_{z}}|^{2}}{\Tr{H_{z}H_{\bar{z}}}},\label{eq:sigmaTilde}
\end{align}
which appear in the asymptotic formulas below.

The complex partial Schur decomposition is the map
\begin{align}
    B(z,\mbf{v},\mbf{w},B')&=R(\mbf{v})\begin{pmatrix}z&\mbf{w}^{*}\\0&B'\end{pmatrix}R(\mbf{v}),
\end{align}
where $R(\mbf{v})$ is the Householder than exchanges $\mbf{v}$ with the first coordinate vector. For $X\in\mbb{M}_{N}(\mbb{C})$ and $z\in\mbb{C}$, define the probability measure on $S^{N-1}(\mbb{C})$ by
\begin{align}
    d\mu^{X}_{z}(\mbf{v})&=\frac{1}{K(z)}\left(\frac{N}{\pi t}\right)^{N-1}e^{-\frac{N}{t}\|X_{z}\mbf{v}\|^{2}}d_{H}\mbf{v},\label{eq:mu}
\end{align}
with normalisation $K(z)$. Let $B=X+\sqrt{t}Y$, where $Y\sim GinUE(N)$ and $X$ is deterministic. Then we have
\begin{align}
    \mbb{E}_{B}\left[\sum_{n=1}^{N}f(z_{n},\mbf{r}_{n},B'_{n})\right]&=\frac{N}{2\pi^{2}t}\int_{\mbb{C}}K(z)\mbb{E}_{\mu_{z}}\left[\mbb{E}_{Y'}\left[f(z,\mbf{v},B')|\det B'_{z}|^{2}\right]\right]dz,\label{eq:complexPartialSchur}
\end{align}
where $B'=X^{(\mbf{v})}+\sqrt{\frac{Nt}{N-1}}Y'$ and $Y'\sim GinUE(N-1)$. Here $f(z,\mbf{v},B')$ is any measurable function such that the integral converges absolutely.

The real-real partial Schur decomposition is the map
\begin{align}
    B(u,\mbf{v},\mbf{w},B')&=R(\mbf{v})\begin{pmatrix}u&\mbf{w}^{T}\\0&B'\end{pmatrix}R(\mbf{v}).
\end{align}
For $X\in\mbb{M}_{N}(\mbb{R})$ and $u\in\mbb{R}$, define the probability measure on $S^{N-1}(\mbb{R})$ by
\begin{align}
    d\nu^{X}_{u}(\mbf{v})&=\frac{1}{K_{\mbb{R}}(u)}\left(\frac{N}{2\pi t}\right)^{N/2-1}e^{-\frac{N}{2t}\|X_{z}\mbf{v}\|^{2}}d_{H}\mbf{v},\label{eq:nu}
\end{align}
with normalisation $K_{\mbb{R}}(u)$. Let $B=X+\sqrt{t}Y$, where $Y\sim GinOE(N)$ and $X$ is deterministic. Then we have
\begin{align}
    \mbb{E}_{B}\left[\sum_{n=1}^{N_{\mbb{R}}}f(u_{n},\mbf{r}_{\mbb{R},n},B'_{\mbb{R},n})\right]&=\frac{N}{4\pi t}\int_{\mbb{R}}K_{\mbb{R}}(u)\mbb{E}_{\nu_{u}}\left[\mbb{E}_{Y'}\left[f(u,\mbf{v},B')|\det B'_{u}|\right]\right]du,\label{eq:real-real}
\end{align}
where $B'=X^{(\mbf{v})}+\sqrt{\frac{Nt}{N-1}}Y'$ and $Y'\sim GinOE(N-1)$.

The real-complex partial Schur decomposition is the map
\begin{align}
    B(x,y,\delta,V,W,B')&=Q(V)\begin{pmatrix}Z&W^{T}\\0&B'\end{pmatrix}Q(V),
\end{align}
where $Q(V)$ is the product of Householders whose first two columns is $V$ and
\begin{align}
    Z&=\begin{pmatrix}x&b\\-c&x\end{pmatrix},\quad y=\sqrt{bc}\geq0,\quad\delta=b-c\geq0.\label{eq:Z}
\end{align}
Define the probability measure on $O(N,2)$ by
\begin{align}
    d\xi^{X}_{\delta,z}(V)&=\frac{1}{L(\delta,z)}\left(\frac{N}{2\pi t}\right)^{N-3}\exp\left\{-\frac{N}{2t}\tr\left(V^{T}X^{T}XV-2Z^{T}V^{T}XV+Z^{T}Z\right)\right\}d_{H}V,
\end{align}
with normalisation $L(\delta,z)$. Then we have
\begin{align}
    \mbb{E}\left[\sum_{n=1}^{N_{\mbb{C}}}f(z_{n},\mbf{r}_{\mbb{C},n},B'_{\mbb{C},n})\right]&=\left(\frac{N}{2\pi t}\right)^{3}\int_{\mbb{C}}\int_{0}^{\infty}\frac{2y\delta}{\sqrt{\delta^{2}+4y^{2}}}L(\delta,z)\nonumber\\
    &\times\mbb{E}_{\xi_{\delta,z}}\left[\mbb{E}_{Y'}\left[f(z,\delta,\mbf{r}(V),B')|\det B'_{z}|^{2}\right]\right]d\delta dz,\label{eq:real-complex}
\end{align}
where $B'=X^{(V)}+\sqrt{\frac{Nt}{N-2}}Y'$, $Y'\sim GinOE(N-2)$,
\begin{align}
    \mbf{r}(V)&=\sqrt{\frac{b}{b+c}}\mbf{v}_{1}+i\sqrt{\frac{c}{b+c}}\mbf{v}_{2},
\end{align}
and $\mbf{v}_{1},\mbf{v}_{2}$ are the columns of $V$.

For proofs and more details about these definitions see \cite{maltsev_bulk_2023,osman_bulk_2024}.

\section{Moments of the Absolute Value of the Characteristic Polynomial}\label{sec4}

Let $m\in\mbb{N}$. Our goal in this section is to calculate
\begin{align}
    D_{2m-1}(u)&:=\mbb{E}_{Y}\left[|\det(X_{u}+\sqrt{t}Y)|^{2m-1}\right],
\end{align}
where $Y\sim GinOE(N)$ and $m\in\mbb{N}$. The even moments will be follow as a byproduct. We write
\begin{align*}
    |\det(X_{u}+\sqrt{t}Y)|^{2m-1}&=\lim_{\eta\to0}\frac{|\det(X_{u}+\sqrt{t}Y)|^{2m}}{\det^{1/2}\left(|X_{u}+\sqrt{t}Y|^{2}+\eta^{2}\right)}
\end{align*}
and take the limit outside the expectation. The determinant in the numerator is represented by an integral over anti-commuting variables:
\begin{align*}
    |\det(X_{u}+\sqrt{t}Y)|^{2m}&=\int\exp\left\{-\sum_{j=1}^{2m}\psi^{*}_{j}(X_{u}+\sqrt{t}Y)\psi_{j}\right\}d\psi.
\end{align*}
The determinant in the denominator is represented by an ordinary integral:
\begin{align*}
    \det^{-1/2}\left(|X_{u}+\sqrt{t}Y|^{2}+\eta^{2}\right)&=\frac{1}{(2\pi)^{N/2}}\int_{\mbb{R}^{N}}\exp\left\{-\frac{1}{2}\mbf{x}^{T}\left(|X_{u}+\sqrt{t}Y|^{2}+\eta^{2}\right)\mbf{x}\right\}d\mbf{x}.
\end{align*}
We collect all the terms in the exponent that depend on $Y$ and compute the expectation:
\begin{align*}
    &\mbb{E}_{Y}\left[\exp\left\{-\frac{t}{2}\mbf{x}^{T}Y^{T}Y\mbf{x}-\sqrt{t}\tr Y\left(\mbf{x}\mbf{x}^{T}X_{u}^{T}-\sum_{j=1}^{2m}\psi_{j}\psi_{j}^{*}\right)\right\}\right]\\
    &=(1+\tr S_{\mbf{x}})^{-N/2}\exp\left\{\frac{1}{2}\frac{\tr S_{\mbf{x}}}{1+\tr S_{\mbf{x}}}\mbf{x}^{T}|X_{u}|^{2}\mbf{x}+\sum_{j=1}^{2m}\psi_{j}^{*}X_{u}S_{\mbf{x}}(1+S_{\mbf{x}})^{-1}\psi_{j}\right.\\
    &\left.+\frac{t}{2N}\sum_{1\leq j\neq k\leq 2m}(\psi_{j}^{*}\bar{\psi}_{k})(\psi_{k}^{T}(1+S_{\mbf{x}})^{-1}\psi_{j})\right\},
\end{align*}
where we have defined $S_{\mbf{x}}=\frac{t}{N}\mbf{x}\mbf{x}^{T}$. The quartic term on the last line is made quadratic by a Hubbard-Stratonovich transformation:
\begin{align*}
    e^{\frac{t}{N}\sum_{j\neq k}^{2m}(\psi_{j}^{*}\bar{\psi}_{k})(\psi_{k}^{T}(1+S_{\mbf{x}})^{-1}\psi_{j})}&=\left(\frac{N}{\pi t}\right)^{m(2m-1)}\int_{\mbb{M}^{skew}_{2m}(\mbb{C})}e^{-\frac{N}{2t}\tr|Q|^{2}+\sum_{j\neq k}^{2m}(q_{jk}\psi_{j}^{*}\bar{\psi}_{k}+\bar{q}_{jk}\psi_{k}^{T}(1+S_{\mbf{x}})^{-1}\psi_{j})}dQ.
\end{align*}
Now the term in the exponent dependent on $\psi_{j}$ is a quadratic form $\Psi^{T}M\Psi$ in the vector $\Psi=(\bar{\psi}_{1},...,\bar{\psi}_{2m},\psi_{1},...,\psi_{2m})^{T}$, where
\begin{align*}
    M&=\begin{pmatrix}Q\otimes1_{N}&1_{2m}\otimes X_{u}(1+S_{\mbf{x}})^{-1}\\-1_{2m}\otimes(1+S_{\mbf{x}})^{-1}X_{u}^{T}&Q^{*}\otimes(1+S_{\mbf{x}})^{-1}\end{pmatrix}.
\end{align*}
Integrating over $\psi_{j}$ we obtain the pfaffian of $M$, which depends only on the singular values $\sigma_{1},...,\sigma_{m}$ of $Q$:
\begin{align*}
    \pf M&=\prod_{j=1}^{m}\det\left(\sigma_{j}^{2}+|X_{u}|^{2}\right)\left(1+\sigma_{j}^{2}\mbf{x}^{T}H_{u}(\sigma_{j})\mbf{x}\right).
\end{align*}

Changing variables 
\begin{align*}
    \mbf{x}&=\sqrt{\frac{Nr}{t}}\mbf{v},\quad,r\in\mbb{R}_{+},\,\mbf{v}\in S^{N-1}(\mbb{R}),\\
    Q&=U\begin{pmatrix}0&\bs\sigma\\-\bs\sigma&0\end{pmatrix}U^{T},\quad\bs\sigma\in\mbb{R}^{m}_{+},\,U\in U(2m)/USp(2)^{m},
\end{align*}
we arrive at the formula
\begin{align}
    D_{2m-1}(u)&=\frac{N}{4\pi t}e^{-N(m-1/2)\phi_{u}}\lim_{\eta\to0}\int_{0}^{\infty}\frac{1}{(1+r)^{2m+1}}e^{-\frac{N}{2t}\eta^{2}r}\psi(r)dr,
\end{align}
where
\begin{align}
    \psi(r)&=\left(\frac{Nr}{2\pi t(1+r)}\right)^{N/2-1}e^{-\frac{N}{2}\phi_{u}}\int_{S^{N-1}(\mbb{R})}e^{-\frac{Nr}{2t(1+r)}\|X_{u}\mbf{v}\|^{2}}I(r,\mbf{v})d_{H}\mbf{v},
\end{align}
and
\begin{align}
    I(r,\mbf{v})&=\frac{2^{m}}{m!\prod_{j=1}^{2m-1}j!}\left(\frac{N}{t}\right)^{m(2m-1)}\int_{\mbb{R}^{m}_{+}}\Delta^{4}(\bs\sigma^{2})\prod_{j=1}^{m}e^{-N\left[\phi_{u}(\sigma_{j})-\phi_{u}\right]}\left(1+r\sigma_{j}^{2}\mbf{v}^{T}H_{u}(\sigma_{j})\mbf{v}\right)\sigma_{j}d\sigma_{j}.
\end{align}
Since the integrand decays as $(1+r)^{m+1}$ we can take the limit $\eta\to0$ inside the integral. Note that $I(r,\mbf{v})$ with $r=0$ gives us the corresponding formula for $D_{2m}(u)$.

By interlacing and \eqref{eq:phiBound} we obtain
\begin{align}
    e^{-N\left[\phi_{u}(\sigma)-\phi_{u}\right]}\mbf{v}^{T}H_{u}(\sigma)\mbf{v}&\leq e^{-\frac{CN(\sigma-\eta_{u})^{2}}{t}}.
\end{align}
With this we restrict $\sigma_{j}$ to the interval $[0,\sqrt{\frac{t}{N}}\log N]$ for each $j=1,...,m$. In this interval we have
\begin{align*}
    \phi_{u}(\sigma)&=\phi_{u}+4\eta_{u}^{2}\Tr{H_{u}^{2}}(\sigma-\eta_{u})^{2}+O\left(\frac{\log^{3} N}{\sqrt{Nt}}\right),\\
    \mbf{v}^{T}H_{u}(\sigma_{j})\mbf{v}&=\left[1+O\left(\frac{\log N}{\sqrt{Nt}}\right)\right]\mbf{v}^{T}H_{u}\mbf{v},\\
    \Delta^{4}(\bs\sigma^{2})&=\left[1+O\left(\frac{\log N}{\sqrt{Nt}}\right)\right](2\eta_{u})^{2m(m-1)}\Delta^{4}(\bs\sigma),
\end{align*}
and so
\begin{align}
    I(r,\mbf{v})&=\left[1+O\left(\frac{\log^{3} N}{\sqrt{Nt}}\right)\right]d_{N,2m}(u)\left(1+r\eta_{u}^{2}\mbf{v}^{T}H_{u}\mbf{v}\right)^{m},
\end{align}
where
\begin{align}
    d_{N,2m}(u)&=\frac{(2\pi)^{m/2}}{\prod_{j=1}^{m-1}(2j)!}\left(\frac{N}{t^{2}\Tr{H_{u}^{2}}}\right)^{m(m-1/2)}.
\end{align}
We have used Selberg's integral to evaluate the integral over $\bs\sigma$. Note that $d_{N,2m}(u)e^{-Nm\phi_{u}}$ is the asymptotic formula for $\mbb{E}_{Y}\left[|\det(X_{u}+\sqrt{t}Y)|^{2m}\right]$.

Now we come to the integral over the sphere. Let us first obtain an upper bound when $r\in(0,1/\delta)$ for some small $\delta>0$. We use the trivial bound $\eta_{u}^{2}\mbf{v}^{T}H_{u}\mbf{v}\leq1$ so that we only have the exponential term, which we treat by the duality formula \cite[Lemma 3.4]{osman_bulk_2024}:
\begin{align}
    \psi(r)&\leq Cd_{N,2m}(u)(1+r)^{m}e^{-\frac{N}{2}\wt{\phi}_{u}(r)}\int_{-\infty}^{\infty}e^{\frac{iNrp}{2t(1+r)}}\det^{-1/2}\left[1+ipH_{u}(\eta_{u}(r))\right]dp,
\end{align}
where
\begin{align}
    \wt{\phi}_{u}(r)&=\phi_{u}-\frac{r\eta_{u}^{2}(r)}{t(1+r)}+\Tr{\log(\eta^{2}_{u}(r)+|X_{u}|^{2})},\label{eq:phiTilde}
\end{align}
and $\eta_{u}(r)$ is defined by
\begin{align}
    t\Tr{H_{u}(\eta_{u}(r))}&=\frac{r}{1+r}.\label{eq:eta_r}
\end{align}

From the definition it follows that $\eta_{u}(r)$ is decreasing in $r\geq0$ and we have the asymptotics
\begin{align}
    \eta_{u}(r)&=\begin{cases}\left[1+O(r^{-1})\right]\eta_{u}&\quad r\gg1\\
    O(\sqrt{t/r})&\quad r\ll \frac{t}{\eta_{u}^{2}+\|X_{u}\|^{2}}
    \end{cases}.
\end{align}
To see this, we observe that the definition of $\eta_{u}(r)$ is equivalent to
\begin{align}
    t[\eta^{2}_{u}(r)-\eta^{2}_{u}]\Tr{H_{u}(\eta_{u}(r))H_{u}}&=\frac{1}{1+r}.
\end{align}
Since $\eta_{u}(r)$ is decreasing in $r$ and $\eta\mapsto\eta^{2}\Tr{H_{u}(\eta)}$ is increasing, we deduce that 
\begin{align}
    \frac{t[\eta^{2}_{u}(r)-\eta^{2}_{u}]\eta^{2}_{u}\Tr{H^{2}_{u}}}{\eta^{2}_{u}(r)}&\leq\frac{1}{1+r}\leq t[\eta^{2}_{u}(r)-\eta^{2}_{u}]\Tr{H^{2}_{u}}.
\end{align}
Noting that $\Tr{H^{2}_{u}}=O(t^{-3})$, this gives us the large $r$ asymptotic. For the small $r$ asymptotic, we note that for large $\eta$ we have $\Tr{H_{u}(\eta)}=O(\eta^{-2})$.

In the region $r\ll t/(\eta_{u}^{2}+\|X_{u}\|^{2})$, the integral over $p$ satisfies the bound
\begin{align*}
    \left|\int_{-\infty}^{\infty}e^{\frac{iNrp}{2t(1+r)}}\det^{-1/2}\left[1+ipH_{u}(\eta_{u}(r))\right]dp\right|&\leq\int_{-\infty}^{\infty}\left(1+\frac{p^{2}}{(\|X_{u}\|^{2}+\eta^{2}_{u}(r))^{2}}\right)^{-N/2}dp\\
    &\leq\frac{C(1+r^{-1})}{\sqrt{N}},
\end{align*}
while $\wt{\phi}_{u}(r)$ satisfies
\begin{align*}
    \wt{\phi}_{u}(r)&\geq-C+\Tr{\log\left[1+(\eta^{2}_{u}(r)-\eta^{2}_{u})H_{u}\right]}\\
    &\geq -C+\Tr{\log\left[1+\frac{Ct}{r}H_{u}\right]}\\
    &\geq -C+\log\frac{t}{r(\eta_{u}^{2}+\|X_{u}\|^{2})}.
\end{align*}
Thus $\psi(r)$ is bounded by
\begin{align*}
    \psi(r)&\leq\frac{Cd_{N,2m}(\eta_{u}^{2}+\|X_{u}\|^{2})}{t\sqrt{N}}e^{CN}\left(\frac{r(\eta_{u}^{2}+\|X_{u}\|^{2})}{t}\right)^{N/2-1},
\end{align*}
and we can neglect the integral over $0<r<\frac{\delta t}{\eta_{u}^{2}+\|X_{u}\|^{2}}$.

When $\frac{\delta t}{\eta_{u}^{2}+\|X_{u}\|^{2}}<r<\frac{1}{\delta}$, we compute the derivative of $\wt{\phi}_{u}(r)$ to be
\begin{align*}
    \partial_{r}\wt{\phi}_{u}(r)&=-\frac{\eta^{2}_{u}(r)}{t(1+r)^{2}},
\end{align*}
and so $\wt{\phi}_{u}(r)$ is monotonically decreasing to $0$ at $r=\infty$. If $r>1/\delta$ for sufficiently small $\delta$, then from the asymptotics of $\eta_{u}(r)$ we have $\eta_{u}(r)=[1+O(\delta)]\eta_{u}$ and
\begin{align*}
    -\partial_{r}\wt{\phi}_{u}(r)&\geq\frac{C\eta^{2}_{u}}{tr^{2}},
\end{align*}
which gives us the lower bound
\begin{align*}
    \wt{\phi}_{u}(1/\delta)&=\wt{\phi}_{u}(\infty)-\int_{1/\delta}^{\infty}\partial_{r}\wt{\phi}_{u}(r)dr\\
    &\geq \frac{C\eta_{u}^{2}}{t}\int_{1/\delta}^{\infty}\frac{1}{r^{2}}dr\\
    &\geq C\delta t.
\end{align*}
Thus we have
\begin{align*}
    e^{-\frac{N}{2}\wt{\phi}_{u}(r)}&\leq e^{-C\delta Nt},\quad r<1/\delta,
\end{align*}
and can restrict $r$ to the region $r>1/\delta$ for $\delta=\frac{\log^{2}N}{Nt}$. 

Having restricted $r$ to this region we return to the full expression for $\psi(r)$. We define the $u$ and $r$-dependent probability measure $\mu_{u,r}$ on $S^{N-1}(\mbb{R})$ by
\begin{align}
    d\mu_{u,r}(\mbf{v})&=\frac{1}{K_{u}(r)}\left(\frac{Nr}{2\pi t(1+r)}\right)^{N/2-1}e^{-\frac{Nr}{2t(1+r)}\|X_{u}\mbf{v}\|^{2}}d_{H}\mbf{v},
\end{align}
where $K_{u}(r)$ is the normalisation. In terms of $\mu_{u,r}$ we have
\begin{align}
    \psi(r)&=\left[1+O\left(\frac{\log^{3} N}{\sqrt{Nt}}\right)\right]d_{N,2m}(u)e^{-\frac{N}{2}\phi_{u}}K_{u}(r)\mbb{E}_{u,r}\left[\left(1+r\eta_{u}^{2}\mbf{v}^{T}H_{u}\mbf{v}\right)^{m}\right].
\end{align}

Using the duality formula \cite[Lemma 3.4]{osman_bulk_2024} for integrals on the sphere we have
\begin{align}
    K_{u}(r)&=e^{\frac{N}{2}\left[\phi_{u}-\wt{\phi}_{u}(r)\right]}\int_{-\infty}^{\infty}e^{\frac{iNrp}{2t(1+r)}}\det^{-1/2}\left[1+ipH_{u}(\eta_{u}(r))\right]dp.
\end{align}
We can estimate the $p$ integral using the fact that $\eta_{u}(r)=O(t)$ and 
\begin{align*}
    |\det^{-1/2}[1+ipH_{u}(\eta_{u}(r))]&\leq\exp\left\{-\frac{CNp^{2}\Tr{H_{u}^{2}(\eta_{u}(r))}}{1+p^{2}/t^{4}}\right\}\\
    &\leq\exp\left\{-CNt\cdot\frac{p^{2}/t^{4}}{1+p^{2}/t^{4}}\right\},
\end{align*}
whereas for large $p>C\|X_{u}\|$ we have $|\det^{-1/2}\left[1+ipH_{u}(\eta_{u}(r))\right]|\leq e^{-CN\log|p|}$. Thus we can restrict $p$ to the interval $|p|<\sqrt{\frac{t^{3}}{N}}\log N$ and then Taylor expand the determinant to obtain
\begin{align}
    K_{u}(r)&=\left[1+O\left(\frac{\log^{3} N}{\sqrt{Nt}}\right)\right]\sqrt{\frac{4\pi}{N\Tr{H_{u}^{2}(\eta_{u}(r))}}}e^{\frac{N}{2}\left[\phi_{u}-\wt{\phi}_{u}(r)\right]}.
\end{align}

To evaluate the expectation of quadratic forms we compute instead the moment generating function
\begin{align*}
    m_{u,r}(\lambda)&:=e^{-\frac{\lambda\eta_{u}^{2}(1+r)t}{r}\Tr{H_{u}(\eta_{u}(r))H_{u}}}\mbb{E}_{u,r}\left[e^{\lambda\eta_{u}^{2}\mbf{v}^{T}H_{u}\mbf{v}}\right]\\
    &=\frac{e^{\frac{N}{2}\left[\phi_{u}-\wt{\phi}_{u}(r)\right]}}{K_{u}(r)}\int_{-\infty}^{\infty}e^{\frac{iNrp}{2t(1+r)}}\det^{-1/2}\left[1+ip\left(\eta_{u}^{2}(r)+|X_{u}|^{2}-\frac{2\lambda\eta_{u}^{2}(1+r)t}{Nr}H_{u}\right)^{-1}\right]dp\\
    &\times e^{-\frac{\lambda\eta_{u}^{2}(1+r)t}{r}\Tr{H_{u}(\eta_{u}(r))H_{u}}}\det^{-1/2}\left[1-\frac{2\lambda\eta_{u}^{2}(1+r)t}{Nr}H_{u}(\eta_{u}(r))H_{u}\right].
\end{align*}
Assume that $|\lambda|<\frac{N\eta_{u}^{2}(r)r}{4t(1+r)}<CNt$. Then
\begin{align*}
    \frac{2|\lambda|\eta_{u}^{2}(1+r)t}{Nr}\left\|H_{u}(\eta_{u}(r))H_{u}\right\|&\leq\frac{1}{2}.
\end{align*}
This means we can treat the integral over $p$ in the same way as before to obtain
\begin{align*}
    m_{u,r}(\lambda)&\leq Ce^{-\frac{\lambda\eta_{u}^{2}(1+r)t}{r}\Tr{H_{u}(\eta_{u}(r))H_{u}}}\det^{-1/2}\left[1-\frac{2\lambda\eta_{u}^{2}(1+r)t}{Nr}H_{u}(\eta_{u}(r))H_{u}\right]\\
    &\leq C\exp\left\{\frac{2\lambda^{2}\eta_{u}^{4}(1+r)^{2}t^{2}}{Nr^{2}}\Tr{H_{u}^{2}(\eta_{u}(r))H_{u}^{2}}\right\}\\
    &\leq Ce^{\frac{C\lambda^{2}}{Nt}}.
\end{align*}
By Markov's inequality we deduce that
\begin{align}
    \eta_{u}^{2}\mbf{v}^{T}H_{u}\mbf{v}&=\left[1+O\left(\frac{\log N}{\sqrt{Nt}}\right)\right]\frac{\eta_{u}^{2}(1+r)t}{r}\Tr{H_{u}(\eta_{u}(r))H_{u}}\\
    &=\left[1+O\left(\frac{\log N}{\sqrt{Nt}}\right)\right]\eta_{u}^{2}t\Tr{H_{u}^{2}},
\end{align}
with probability $1-O(e^{-C\log^{2}N})$. Therefore we can replace the quadratic form with its deterministic approximation to obtain
\begin{align}
    \psi(r)&=\left[1+O\left(\frac{\log^{3} N}{\sqrt{Nt}}\right)\right]\sqrt{\frac{4\pi}{N\Tr{H_{u}^{2}}}}d_{N,2m}(u)e^{-\frac{N}{2}\wt{\phi}_{u}(r)}\left(r\eta_{u}^{2}t\Tr{H_{u}^{2}}\right)^{m}.
\end{align}
We can approximate $\wt{\phi}_{u}(r)$ by approximating its derivative
\begin{align*}
    \partial_{r}\wt{\phi}_{u}(r)&=-\left[1+O\left(\frac{\log^{2}N}{Nt}\right)\right]\frac{\eta_{u}^{2}}{tr^{2}}.
\end{align*}
Inserting this into $D_{m}(u)$ we obtain
\begin{align*}
    D_{2m-1}(u)&=\left[1+O\left(\frac{\log^{2} N}{\sqrt{Nt}}\right)\right]\sqrt{\frac{4\pi}{N\Tr{H_{u}^{2}}}}d_{N,2m}(u)(\eta_{u}^{2}t\Tr{H_{u}^{2}})^{m}e^{-N(m-1/2)\phi_{u}}\\
    &\times\frac{N}{4\pi t}\int_{Nt/\log^{2}N}^{\infty}\frac{1}{r^{m+1}}e^{-\frac{N\eta_{u}^{2}}{2tr}}dr\\
    &=\left[1+O\left(\frac{\log^{2}N}{\sqrt{Nt}}\right)\right]\frac{(m-1)!2^{m-1}d_{N,2m}(u)}{\sqrt{\pi}}\left(\frac{t^{2}\Tr{H_{u}^{2}}}{N}\right)^{m-1/2}e^{-N(m-1/2)\phi_{u}}.
\end{align*}
Now we note that (this follows from the duplication formula for the Gamma function and the functional equation $G(z+1)=\Gamma(z)G(z)$, see \cite[eq. (C.5)]{serebryakov_schur_2023})
\begin{align*}
    \prod_{j=1}^{m-1}(2j)!&=\frac{2^{m(m-1)}}{\pi^{m/2}}\frac{G(m+1/2)G(m+1)}{G(1/2)},
\end{align*}
where $G(z)$ is the Barnes G-function. Therefore we can extend $d_{N,2m}$ to all real $m$ to obtain
\begin{align*}
    D_{2m-1}(u)&=\left[1+O\left(\frac{\log^{2}N}{\sqrt{Nt}}\right)\right]d_{N,2m-1}(u)e^{-N(m-1/2)\phi_{u}}.
\end{align*}

\section{Proof of \Cref{thm1}}\label{sec5}
Throughout this section we set $\eta=N^{-1-\epsilon}$. We begin by reducing the problem to obtaining a bound for the expectation value of a trace of the resolvent of the Hermitisation. We expect that there are a finite number of eigenvalues in a ball of radius $N^{-1/2}$ centred at $z_{0}$, which suggests that we can afford to take a union bound to estimate the minimum:
\begin{align}
    P\left(\min_{\sqrt{N}|z_{n}-z_{0}|<r}s_{N-1}(z_{n})<\eta\right)&\leq \sum_{n=1}^{N}P\left(s_{N-1}(z_{n})<\eta,\,|z_{n}-z_{0}|<rN^{-1/2}\right).
\end{align}
Now we note that if $s_{N-1}(z_{n})<\eta$, then we have
\begin{align}
    f_{\eta}(z_{n})&:=\eta^{2}\tr H_{z_{n}}(\eta)-1=\sum_{m=1}^{N-1}\frac{\eta^{2}}{s_{m}^{2}(z_{n})+\eta^{2}}\geq\frac{1}{2}.
\end{align}
If $g_{z_{0}}(z)=g\left(\frac{\sqrt{N}(z-z_{0})}{r}\right)$ and $g$ is a smooth function such that $g(z)=1$ for $|z|<1$ and $g(z)=0$ for $|z|>2$, then by Markov's inequality we obtain
\begin{align}
    P\left(\min_{\sqrt{N}|z_{n}-z_{0}|<r}s_{N-1}(z_{n})<\eta\right)&\leq2\mbb{E}\left[\sum_{n=1}^{N}f_{\eta}(z_{n})g_{z_{0}}\left(z_{n}\right)\right].
\end{align}
Now that we have a sum over eigenvalues, we can use Girko's formula to reduce to the Gauss-divisible case.
\begin{lemma}\label{lem:comp}
Let $t>0$ and $A$ and $B$ be $t$-matching. Then for any $\xi>0$ we have
\begin{align}
    \left|\mbb{E}_{A}\left[\sum_{n=1}^{N}f_{\eta}(z_{n})g_{z_{0}}(z_{n})\right]-\mbb{E}_{B}\left[\sum_{n=1}^{N}f_{\eta}(z_{n})g_{z_{0}}(z_{n})\right]\right|&\leq N\eta+N^{\xi}\left(\frac{t}{(N\eta)^{36}}+\frac{1}{N^{1/2}(N\eta)^{42}}\right).\label{eq:comp}
\end{align}
\end{lemma}
\begin{proof}
Let 
\begin{align}
    \mc{L}&:=\sum_{n=1}^{N}f_{\eta}(z_{n})g_{z_{0}}(z_{n}).
\end{align}
Note that
\begin{align*}
    f_{\eta}(z)&=\eta\Im\tr G_{z}(i\eta)-1,\\
    \partial_{z}f_{\eta}(z)&=\eta\Im\tr G_{z}^{2}(i\eta)F,\\
    \Delta_{z}f_{\eta}(z)&=\eta\Im\tr\left(G^{2}_{z}(i\eta)FG_{z}(i\eta)F^{*}+G_{z}(i\eta)FG^{2}_{z}(i\eta)F^{*}\right).
\end{align*}
If we cover the support of $g_{z_{0}}$ by disks of radius $\eta/N^{2}$, then by \Cref{lem:aPriori} and a union bound we can ensure that for any $\xi,D>0$ the event
\begin{align}
    \mc{E}_{\xi}&=\bigcap_{|z-z_{0}|<rN^{-1/2}}\left\{|f_{\eta}(z)|<N^{\xi},\,|\partial_{z}f_{\eta}(z)|<\frac{N^{\xi}}{N^{1/2}\eta},\,|\Delta_{z}f_{\eta}(z)|<\frac{N^{\xi}}{N\eta^{2}}\right\}
\end{align}
holds with probability $1-N^{-D}$. Therefore we have
\begin{align}
    \mbb{E}\left[\mc{L}\right]&=\mbb{E}\left[1_{\mc{E}_{\xi}}\mc{L}\right]+O(N^{-D}).
\end{align}
Since $f_{\eta}$ and $g$ are $C^{2}$ we can apply Girko's formula to the sum over $z_{n}$:
\begin{align}
    \mc{L}&=-\frac{1}{4\pi}\int_{\mbb{C}}\Delta_{z}\left(f_{\eta}(z)g_{z_{0}}\left(z\right)\right)\int_{0}^{N^{D}}\Im\tr G_{z}(i\sigma)d\sigma dz+O(N^{-D}).
\end{align}
The error term holds in the sense of stochastic domination and follows from the fact that $\|A\|\leq C$ with probability $1-N^{-D}$ and $\|\Delta_{z}(f_{\eta}g_{z_{0}})\|_{1}\leq CN^{2+2\epsilon}$ deterministically. We stress that $f_{\eta}$ is a \textit{random} function depending on $A$, unlike in the usual applications of Girko's formula. Therefore the expectation acts on both $f_{\eta}$ and $\Im\tr G_{z}(i\sigma)$.

We fix a $\nu>0$ and $\eta_{1}=N^{-1-\nu}$ to be specified later and split the integral over $\sigma$ into three regimes:
\begin{align}
    \mc{L}&=-\frac{1}{4\pi}\left[I(0,\eta_{1})+I(\eta_{1},N^{D})\right]+O(N^{-D}),\\
    I(\eta_{1},\eta_{2})&:=\int_{\mbb{C}}\Delta_{z}\left(f_{\eta}(z)g_{z_{0}}(z)\right)\int_{\eta_{1}}^{\eta_{2}}\Im\tr G_{z}(i\sigma)d\sigma dz.
\end{align}
We estimate the small $\sigma$ integral as in \cite[Section 4]{cipolloni_edge_2021} and \cite[Section 5.2]{he_edge_2023}. First we note that $I(0,\eta)$ satisfies a deterministic bound
\begin{align}
    |I(0,\eta_{1})|&\leq 4\pi\mc{L}+|I(\eta_{1},\infty)|\\
    &\leq CN\log N.
\end{align}
Then we remove disks of radius $\rho$ centred at the eigenvalues $z_{n}$ from the $z$-integral and use the bound
\begin{align*}
    \left|\int_{B_{\rho}(z_{n})}\Delta_{z}(f_{\eta}g_{z_{0}})(z)\log\left(1+\frac{\eta_{1}^{2}}{s_{n}^{2}(z)}\right)dz\right|&\leq C\|\Delta_{z}(f_{\eta}g_{z_{0}})\|_{\infty}\rho^{2}(|\log \eta|+|\log \rho|),
\end{align*}
to obtain
\begin{align}
    I(0,\eta_{1})&=\int_{\mc{D}}\Delta_{z}(f_{\eta}g_{z_{0}})(z)\sum_{n=1}^{N}\log\left(1+\frac{\eta_{1}^{2}}{s^{2}_{n}(z)}\right) dz+O\left(\frac{N^{2}\rho^{2}(\log N+|\log \rho|)}{\eta^{2}}\right),
\end{align} 
where $\mc{D}=\mbb{C}\setminus\cup_{n}B_{\rho}(z_{n})$ and the error term holds in the sense of stochastic domination. Now cover $\mc{D}\cap\textup{supp}\,g_{z_{0}}$ by disks of radius $\delta$ centred at $w_{j}\in\mc{N}_{\delta}$ and approrhomate the integral by a Riemann sum to obtain
\begin{align}
    I(0,\eta_{1})&=\pi\delta^{2}\sum_{w_{j}\in\mc{N}_{\delta}}\sum_{n=1}^{N}\Delta_{z}\left(f_{\eta}g_{z_{0}}\right)(w_{j})\log\left(1+\frac{\eta_{1}^{2}}{s^{2}_{n}(w_{j})}\right)+O\left(\frac{N^{2}\delta}{\eta^{3}\rho}\right).
\end{align}
If we choose $\rho=\eta^{3/2}/\sqrt{N}$ and $\delta=\rho\eta^{4}/N$ then the error is $O(N\eta\log N)$. Instead of a Riemann sum we could have approximated the integral by the Monte Carlo method as done by Tao--Vu \cite{tao_random_2015}.

Using \Cref{prop:leastSV} and a union bound we have
\begin{align}
    P\left(\min_{w_{j}\in\mc{N}_{\delta}}s_{N}(w_{j})<N^{-L}\right)&\leq \frac{1}{\delta^{2}N^{L/2}}.
\end{align}
Using the deterministic bound on $I(0,\eta_{1})$ we can restrict to this event, which we denote by $\mc{F}$. Recalling the bounds that define the event $\mc{E}_{\xi}$ we have
\begin{align*}
    \mbb{E}\left[1_{\mc{E}_{\xi}}I(0,\eta_{1})\right]&\leq\frac{N^{\xi}}{N^{2}\eta^{2}}\max_{w_{j}\in\mc{N}_{\delta}}\mbb{E}\left[1_{\mc{F}}\sum_{n=1}^{N}\log\left(1+\frac{\eta_{1}^{2}}{s_{n}^{2}(w_{j})}\right)\right].
\end{align*}
We split the sum over $s_{n}$ into two:
\begin{align}
    \mbb{E}\left[1_{\mc{F}}\sum_{n=1}^{N}\log\left(1+\frac{\eta_{1}^{2}}{s_{n}^{2}(w_{j})}\right)\right]&\leq\mbb{E}\left[1_{\mc{F}}\sum_{N^{-L}<s_{n}(w_{j})<N^{\nu/2}\eta_{1}}\log\left(1+\frac{\eta_{1}^{2}}{s_{n}^{2}(w_{j})}\right)\right]\label{eq:E1}\\
    &+\mbb{E}\left[1_{\mc{F}}\sum_{s_{n}(w_{j})>N^{\nu/2}\eta_{1}}\log\left(1+\frac{\eta_{1}^{2}}{s^{2}_{n}(w_{j})}\right)\right]\label{eq:E2}.
\end{align}
For the small singular values we use the fact that $|\{n:s_{n}(w_{j})<N^{\nu/2}\eta_{1}\}|<N^{\xi/2}$ with probability $1-N^{-D}$ and $P(s_{N}(w_{j})<N^{\nu/2}\eta_{1})<N^{-\nu/2}$ to obtain
\begin{align}
    \eqref{eq:E1}&\leq CN^{\xi/2}\log N\cdot P(s_{N}(w_{j})<N^{\nu/2}\eta_{1})\nonumber\\
    &\leq CN^{(\xi-\nu)/2}.
\end{align}
For the singular values greater than $N^{\nu/2}\eta_{1}$, we observe that $N^{\nu}\eta_{1}=1$ and use \Cref{lem:aPriori}:
\begin{align*}
    \eqref{eq:E2}&\leq 2\eta_{1}\mbb{E}\left[\sum_{s_{n}(w_{j})>N^{\nu/2}\eta_{1}}\frac{N^{\nu/2}\eta_{1}}{s_{n}^{2}(w_{j})+(N^{\nu}\eta_{1})^{2}}\right]\nonumber\\
    &\leq N^{1+\xi}\eta_{1}\\
    &\leq N^{\xi-\nu}.
\end{align*}
Thus we find (replacing multiples of $\xi$ as necessary)
\begin{align*}
    \mbb{E}\left[1_{\mc{E}_{\xi}}I(0,\eta_{1})\right]&\leq\frac{N^{\xi-\nu/2}}{N^{2}\eta^{2}}.
\end{align*}
If we choose $\nu=6\epsilon+2\xi$ (i.e. $\eta_{1}=N^{5-2\xi}\eta^{6}$) then we have
\begin{align*}
    \mbb{E}\left[1_{\mc{E}_{\xi}}I(0,\eta_{1})\right]&\leq N\eta.
\end{align*}

In the large $\sigma$ regime we integrate by parts:
\begin{align}
    I(\eta_{1},N^{D})&=\int_{\mbb{C}}f_{\eta}(z)g_{z_{0}}(z)\Im\int_{\eta_{1}}^{N^{D}}\Delta_{z}\tr G_{z}(i\sigma)d\sigma dz\nonumber\\
    &=\int_{\mbb{C}}f_{\eta}(z)g_{z_{0}}(z)\Im\int_{\eta_{1}}^{\infty}\tr\left(G^{2}_{z}(i\sigma)FG_{z}(i\sigma)F^{*}+G_{z}(i\sigma)FG^{2}_{z}(i\sigma)F^{*}\right)d\sigma dz\nonumber+O(N^{-D})\\
    &=-\int_{\mbb{C}}f_{\eta}(z)g_{z_{0}}(z)\Re\int_{\eta_{1}}^{\infty}\partial_{\sigma}\tr G_{z}(i\sigma)FG_{z}(i\sigma)F^{*} d\sigma dz\nonumber+O(N^{-D})\\
    &=\Re\int_{\mbb{C}}f_{\eta}(z)g_{z_{0}}(z)\tr G_{z}(i\eta_{1})FG_{z}(i\eta_{1})F^{*}dz\nonumber+O(N^{-D})\\
    &=:\Re\int_{\mbb{C}}f_{\eta}(z)g_{z_{0}}(z)h_{\eta_{1}}(z)dz+O(N^{-D}).
\end{align}
Since $|f_{\eta}(z)h_{\eta_{1}}(z)|\leq N^{2}/\eta^{2}_{1}$, we can remove the indicator of the event $\mc{E}_{\xi}$:
\begin{align}
    \left|\mbb{E}_{A}\left[\mc{L}\right]-\mbb{E}_{B}\left[\mc{L}\right]\right|&\leq\int_{\mbb{C}}g_{z_{0}}(z)\left|\mbb{E}_{A}\left[f_{\eta}(z)h_{\eta_{1}}(z)\right]-\mbb{E}_{B}\left[f_{\eta}(z)h_{\eta_{1}}(z)\right]\right|dz+CN\eta.
\end{align}
We are now in a position to apply the Lindenberg strategy. Since this is standard we will only give a sketch. It is enough to consider the difference in expectation value when $A$ and $B$ differ in one entry. Assume that this is the entry $(j,k)$ and let $A^{(0)},A^{(1)},A^{(2)}$ denote the matrices whose $(j,k)$ entry is $0,a_{jk},b_{jk}$ respectively. Then we have
\begin{align}
    G^{(1)}_{z}(i\eta)&=\sum_{p=0}^{4}a_{jk}^{p}(G^{(0)}_{z}(i\eta)\Delta_{jk})^{p}G^{(0)}_{z}(i\eta)+a_{jk}^{5}(G^{(0)}_{z}(i\eta)\Delta_{jk})^{5}G^{(1)}_{z}(i\eta),\\
    G^{(2)}_{z}(i\eta)&=\sum_{p=0}^{4}b_{jk}^{p}(G^{(0)}_{z}(i\eta)\Delta_{jk})^{p}G^{(0)}_{z}(i\eta)+b_{jk}^{5}(G^{(0)}_{z}(i\eta)\Delta_{jk})^{5}G^{(2)}_{z}(i\eta),
\end{align}
where $\Delta_{jk}$ is the matrix with -1 in the entry $(j,N+k)$ and zero elsewhere. Thus we have 
\begin{align}
    \tr(X\Delta_{jk})^{p}Y&=(X_{N+k,j})^{p-1}(YX)_{N+k,j}.
\end{align}
We insert these expansions into $f_{\eta}(z)$ and $h_{\eta}(z)$ to obtain polynomials in $a_{jk}$ or $b_{jk}$. Using the deterministic bound $\|G_{z}(i\eta)\|<\eta^{-1}$ and \Cref{cor} we restrict to the event that the estimates in \Cref{lem:aPriori} hold for $G^{(0)}$ and then use these estimates and the bound $|a_{jk}|\prec N^{-1/2}$ to truncate to fourth order. Let us illustrate this with $h_{\eta_{1}}$, for which we have
\begin{align*}
    h_{\eta_{1}}(z)&=\sum_{p,q=0}^{4}a_{jk}^{p+q}\tr(G^{(0)}\Delta_{jk})^{p}G^{(0)}F(G^{(0)}\Delta_{jk})^{q}G^{(0)}F^{*}\\
    &+\sum_{p=0}^{4}a_{jk}^{5+p}\tr(G^{(0)}\Delta_{jk})^{p}G^{(0)}F(G^{(0)}\Delta_{jk})^{5}G^{(1)}F^{*}\\
    &+\sum_{p=0}^{4}a_{jk}^{5+p}\tr(G^{(0)}\Delta_{jk})^{5}G^{(1)}F(G^{(0)}\Delta_{jk})^{p}G^{(0)}F^{*}\\
    &+a_{jk}^{10}\tr(G^{(0)}\Delta_{jk})^{5}G^{(1)}F(G^{(0)}\Delta_{jk})^{5}G^{(1)}F^{*}.
\end{align*}
Consider the expectation of a term in the second line. If $p>0$ then we use \eqref{eq:ext4} and \eqref{eq:ext5}:
\begin{align*}
    &\left|\mbb{E}\left[a_{jk}^{5+p}f^{(0)}_{\eta}(z)\tr (G^{(0)}\Delta_{jk})^{p}G^{(0)}F(G^{(0)}\Delta_{jk})^{5}G^{(1)}F^{*}\right]\right|\\
    &\leq \frac{CN^{\xi}}{N^{(5+p)/2}}\mbb{E}\left[\left|(G^{(0)}_{N+k,j})^{3+p}(G^{(0)}FG^{(0)})_{N+k,j}(G^{(1)}F^{*}G^{(0)})_{N+k,j}\right|\right]\\
    &\leq\frac{CN^{\xi}}{N^{(3+p)/2}(N\eta_{1})^{7+p}},
\end{align*}
where we use another resolvent expansion for the term $(G^{(1)}F^{*}G^{(0)})_{N+k,j}$. If $p=0$, then we use \eqref{eq:ext4} and \eqref{eq:ext6}:
\begin{align*}
    &\left|\mbb{E}\left[a_{jk}^{5}f^{(0)}_{\eta}(z)\tr (G^{(0)}\Delta_{jk})^{5}G^{(1)}F^{*}G^{(0)}F\right]\right|\\
    &\leq\frac{CN^{\xi}}{N^{5/2}}\mbb{E}\left[\left|(G^{(0)}_{N+k,j})^{4}(G^{(1)}F^{*}G^{(0)}FG^{(0)})_{N+k,j}\right|\right]\\
    &\leq\frac{CN^{\xi}}{N^{3/2}(N\eta_{1})^{7}},
\end{align*}
where again we use another resolvent expansion for $G^{(1)}$.

Thus the difference in expectation values due to $h_{\eta_{1}}$ is
\begin{align}
    &\left|\mbb{E}_{A}\left[f^{(0)}_{\eta}(z)h_{\eta_{1}}(z)\right]-\mbb{E}_{B}\left[f^{(0)}_{\eta}(z)h_{\eta_{1}}(z)\right]\right|\nonumber\\
    &=t\sum_{p+q=4}\mbb{E}\left[f^{(0)}_{\eta}(z)\tr(G^{(0)}\Delta_{jk})^{p}G^{(0)}F(G^{(0)}\Delta_{jk})^{q}G^{(0)}F^{*}\right]\nonumber\\
    &+O\left(\frac{N^{\xi}}{N^{3/2}(N\eta_{1})^{7}}\right).
\end{align}
Proceeding in this manner to estimate the remaining terms that contribute to the difference in expectation values and summing over all $N^{2}$ elements, we find
\begin{align}
    \left|\mbb{E}_{A}\left[f_{\eta}(z)h_{\eta_{1}}(z)\right]-\mbb{E}_{B}\left[f_{\eta}(z)h_{\eta_{1}}(z)\right]\right|&\leq\frac{N^{1+\xi}t}{(N\eta_{1})^{6}}+\frac{N^{1/2+\xi}}{(N\eta_{1})^{7}}.
\end{align}
Since $\|g_{z_{0}}\|_{1}\leq CN^{-1}$, $\eta_{1}=N^{5-2\xi}\eta^{6}$ and $\xi$ is arbitrary, we obtain \eqref{eq:comp}.
\end{proof}

In the Gauss-divisible case, we can use the partial Schur decomposition to bound the expectation value. 
\begin{lemma}\label{lem:GaussDivisible}
Let $\tau\in(0,1/3)$, $t=N^{-1/3+\tau}$ and $B=X+\sqrt{t}Y$, where $X$ is a non-Hermitian Wigner matrix and $Y$ is a Ginibre matrix. Then for any $\omega>0$ and $z_{0}\in\mbb{C},\,u_{0}\in\mbb{R}$ such that $|z_{0}|,|u_{0}|<1-\omega$ we have
\begin{align}
    \mbb{E}_{B}\left[\sum_{n=1}^{N}f_{\eta}(z_{n})g_{z_{0}}(z_{n})\right]&\leq CN^{2}\eta^{2},
\end{align}
for complex matrices, and
\begin{align}
    \mbb{E}_{B}\left[\sum_{n=1}^{N_{\mbb{C}}}f_{\eta}(z_{n})g_{z_{0}}(z_{n})\right]&\leq CN^{2}\eta^{2},\\
    \mbb{E}_{B}\left[\sum_{n=1}^{N_{\mbb{R}}}f_{\eta}(u_{n})g_{u_{0}}(u_{n})\right]&\leq CN^{2}\eta^{2}\log(N\eta),
\end{align}
for real matrices.
\end{lemma}
In the case of complex matrices or real eigenvalues of real matrices we can reduce $t$ to $t\geq N^{-1+\tau}$ using the improved local laws in \Cref{prop:ll2}, since in these cases we only need to estimate traces of expressions involving at most two resolvents.

We defer the proof of \Cref{lem:GaussDivisible} to the next section and conclude the proof of \Cref{thm1}. Let $A$ be a complex non-Hermitian Wigner matrix. For $t=N^{-1/3+\tau}$ there exists a non-Hermitian Wigner matrix $\wt{A}$ such that $A$ and $B=\frac{1}{\sqrt{1+t}}\left(\wt{A}+\sqrt{t}Y\right)$ are $t$-matching, where $Y$ is a complex Ginibre matrix. Then using \Cref{lem:comp,lem:GaussDivisible} we find
\begin{align}
    P\left(\min_{\sqrt{N}|z_{n}-z_{0}|<r}s_{N-1}(z_{n})<\eta\right)&\leq\mbb{E}_{A}\left[\sum_{n=1}^{N}f_{\eta}(z_{n})g_{z_{0}}(z_{n})\right]\nonumber\\
    &\leq\mbb{E}_{B}\left[\sum_{n=1}^{N}f_{\eta}(z_{n})g_{z_{0}}(z_{n})\right]+N\eta+\frac{N^{\xi}t}{(N\eta)^{36}}\nonumber\\
    &\leq CN^{2}\eta^{2}+N\eta+\frac{N^{\xi}t}{(N\eta)^{36}}\nonumber\\
    &\leq N\eta+\frac{N^{\xi}t}{(N\eta)^{36}}.
\end{align}
The real case follows in the same way.

\section{Proof of \Cref{lem:GaussDivisible}}\label{sec6}
The proof is based on explicit calculations similar to those in \cite{maltsev_bulk_2023} and \cite{osman_bulk_2024} and in \Cref{sec4}. We will skip the details in some steps and refer to \cite{maltsev_bulk_2023} and \cite{osman_bulk_2024}. We have to consider separately three cases: i) complex entries; ii) real entries and complex eigenvalues; iii) real entries and real eigenvalues.

\subsection{Complex Entries}
Let $\mc{E}$ denote the event that $X$ satisfies the local laws in \Cref{prop:ll1,prop:ll2}. Then $P(\mc{E}^{c})\leq N^{-D}$ for any $D$ and so
\begin{align}
    \mbb{E}_{B}\left[\sum_{n=1}^{N}f_{\eta}(z_{n})g_{z_{0}}(z_{n})\right]&=\mbb{E}_{B}\left[1_{\mc{E}}\sum_{n=1}^{N}f_{\eta}(z_{n})g_{z_{0}}(z_{n})\right]+N^{1-D}.
\end{align}
Henceforth we restrict ourselves to $X\in\mc{E}$.

From \eqref{eq:complexPartialSchur} we have
\begin{align}
    \mbb{E}_{Y}\left[\sum_{n=1}^{N}f_{\eta}(z_{n})g_{z_{0}}(z_{n})\right]&=\frac{N}{2\pi^{2}t}\int_{\mbb{C}}g_{z_{0}}(z)K(z)\mbb{E}_{\mu_{z}}\left[\mbb{E}_{Y'}\left[f_{\eta}(z)\det|B_{z}|^{2}\right]\right]dz,\label{eq:EY}
\end{align}
where $B'=X^{(\mbf{v})}+\sqrt{\frac{(N-1)t}{N}}Y'$ and $Y'$ is a complex Ginibre matrix of size $N-1$.

Let us first calculate the expectation over $Y'$. Expressed in terms of $B'$ the function $f_{\eta}$ takes the form
\begin{align}
    f_{\eta}(z)&=\eta^{2}\tr(|B'_{z}|^{2}+\mbf{w}\mbf{w}^{*}+\eta^{2})^{-1}\nonumber\\
    &\leq\eta^{2}\tr|B'_{z}|^{-2}.\label{eq:fBound}
\end{align}
We use the crude bound on the last line because the resulting integral is simpler to analyse. We anticipate that if $\eta\ll N^{-1}$ then not much is lost in doing so.

Replacing $f_{\eta}$ with its upper bound in \eqref{eq:fBound} we can write
\begin{align*}
    f_{\eta}(z)\det|B'_{z}|^{2}&\leq\eta^{2}\tr|B'_{z}|^{-2}\cdot\det|B'_{z}|^{2}\\
    &=\eta^{2}\lim_{\eta'\to0}\frac{1}{2\eta'}\frac{\partial}{\partial\eta'}\det\left(|B'_{z}|^{2}+\eta'^{2}\right).
\end{align*}
The function inside the limit on the second line is integrable with respect to $e^{-\frac{N}{t}\tr|B'-X^{(\mbf{v})}|^{2}}dB'$ uniformly in $0\leq\eta'\leq C$ and so we can take the limit outside the integral. By integration over anti-commuting variables we obtain
\begin{align}
    \mbb{E}_{Y'}\left[f_{\eta}(z)\det|B_{z}|^{2}\right]&\leq\frac{N\eta^{2}}{\pi t}\lim_{\eta'\to0}\frac{1}{2\eta'}\frac{\partial}{\partial\eta'}\int_{\mbb{C}}e^{-\frac{N}{t}|q+\eta'|^{2}}\det\left(|q|^{2}+|X^{(\mbf{v})}_{z}|^{2}\right)dq.\label{eq:expectation}
\end{align}
Writing $q=\sigma e^{i\theta}$ and integrating over $\theta\in[0,2\pi)$ the right hand side becomes
\begin{align}
    \frac{N\eta^{2}}{t}\lim_{\eta'\to0}\int_{0}^{\infty}e^{-\frac{N}{t}(\sigma^{2}+\eta'^{2})}\det\left(\sigma^{2}+|X^{(\mbf{v})}_{z}|^{2}\right)\left[\frac{2N\sigma }{t\eta'}I_{1}\left(\frac{2N\eta'\sigma}{t}\right)-\frac{N}{t}\right]\sigma d\sigma,
\end{align}
where $I_{1}(x)$ is the modified Bessel function of the second kind. Since $I_{1}(x)\sim x/2$ as $x\to0$, taking the limit inside the integral we obtain
\begin{align}
    \mbb{E}_{Y'}\left[f_{\eta}(z)\det|B'_{z}|^{2}\right]&\leq\frac{2N^{3}\eta^{2}}{t^{3}}\int_{0}^{\infty}e^{-N\phi^{(\mbf{v})}_{z}(\sigma)}\sigma^{3}d\sigma.
\end{align}
This is essentially the same integral that gives the average eigenvalue density. Using \eqref{eq:phiBound} we can restrict $\sigma$ to an $O\left(\sqrt{\frac{t}{N}}\log N\right)$ neighbourhood of the point $\eta^{(\mbf{v})}_{z}$, which gives
\begin{align}
    \mbb{E}_{Y'}\left[f_{\eta}(z)\det|B'_{z}|^{2}\right]&\leq CN^{5/2}t^{1/2}\eta^{2}e^{-N\phi^{(\mbf{v})}_{z}}.
\end{align}

From \cite[Section 7]{maltsev_bulk_2023} we have
\begin{align}
    K(z)&=\left[1+O\left(\frac{\log^{3}N}{\sqrt{Nt}}\right)\right]\sqrt{\frac{2\pi}{N\Tr{H^{2}_{z}}}}e^{N\phi_{z}}\leq\frac{Ct^{3/2}}{N^{1/2}}e^{N\phi_{z}}.
\end{align}
By the above bound, the bound in \eqref{eq:etaBound2} and Cramer's rule we have
\begin{align*}
    K(z)\mbb{E}_{Y'}\left[f_{\eta}(z)\det|B'_{z}|^{2}\right]&\leq C\left|\det(1_{2}\otimes\mbf{v}^{*})G_{z}(1_{2}\otimes\mbf{v})\right|\\
    &=C\left[\eta_{z}^{2}(\mbf{v}^{*}H_{z}\mbf{v})(\mbf{v}^{*}\wt{H}_{z}\mbf{v})+\left|\mbf{v}^{*}X_{z}H_{z}\mbf{v}\right|^{2}\right]
\end{align*}
In \cite{maltsev_bulk_2023} we argued that such quadratic forms concentrate with respect to $\mu_{z}$. Since we only have quadratic terms we can instead evaluate the expectation directly:
\begin{align*}
    \mbb{E}_{\mu_{z}}\left[(\mbf{v}^{*}F_{1}\mbf{v})(\mbf{v}^{*}F_{2}\mbf{v})\right]&=\frac{e^{\frac{N}{t}\eta_{z}^{2}}}{K(z)}\int_{-\infty}^{\infty}e^{\frac{iNp}{t}}\int_{\mbb{C}^{N}}e^{-\frac{N}{t}\mbf{x}^{*}\left(\eta_{z}^{2}+|X_{z}|^{2}+ip\right)\mbf{x}}(\mbf{x}^{*}F_{1}\mbf{x})(\mbf{x}^{*}F_{2}\mbf{x})d\mbf{x}dp\\
    &=\frac{1}{K(z)}\left(\frac{N}{\pi t}\right)^{N-1}\int_{-\infty}^{\infty}e^{\frac{iNp}{t}}\det^{-1}\left(\eta_{z}^{2}+|X_{z}|^{2}+ip\right)\\
    &\times\left[t^{2}\Tr{H_{z}(w_{p})F_{1}}\Tr{H_{z}(w_{p})F_{2}}+\frac{t^{2}}{N}\Tr{H_{z}(w_{p})F_{1}H_{z}(w_{p})F_{2}}\right]dp,
\end{align*}
where $w_{p}=\sqrt{\eta^{2}_{z}+ip}$. The extra terms do not affect the large $p$ behaviour so using the same argument as in \cite{maltsev_bulk_2023} for $K$ we can restrict $p$ to the region $|p|<\sqrt{\frac{t^{3}}{N}}\log N$. The traces can then be estimated using local laws. For example, with $F_{1}=X_{z}H_{z}$ and $F_{2}=H_{z}X_{z}^{*}$ we have
\begin{align*}
    \left|\frac{t^{2}}{N}\Tr{H_{z}(w_{p})F_{1}H_{z}(w_{p})F_{2}}\right|&=\left|\frac{t^{2}}{N}\Tr{(1+ipH_{z})^{-1}H^{1/2}_{z}X_{z}H_{z}^{3/2}(1+ipH_{z})^{-1}H^{3/2}_{z}X_{z}^{*}H^{1/2}_{z}}\right|\\
    &\leq\frac{Ct^{2}}{N}\Tr{H_{z}X_{z}H_{z}^{3}X_{z}^{*}}\\
    &=\frac{Ct^{2}}{N}\Tr{H_{z}\wt{H}_{z}^{2}(1-\eta^{2}\wt{H}_{z})}\\
    &\leq\frac{Ct^{2}}{N}\Tr{H_{z}\wt{H}_{z}^{2}}\\
    &\leq\frac{C}{N}\Tr{H_{z}\wt{H}_{z}}\\
    &\leq\frac{C}{Nt^{2}},
\end{align*}
when $|p|<\sqrt{\frac{t^{3}}{N}}\log N$, using the two-resolvent local law \Cref{prop:ll2}. This is what allows us to take $t=N^{-1+\epsilon}$ instead of $t=N^{-1/3+\epsilon}$.

Estimating the remaining terms in this manner we find 
\begin{align}
    K(z)\mbb{E}_{\mu_{z}}\left[\mbb{E}_{Y'}\left[f_{\eta}(z)\det|B'_{z}|^{2}\right]\right]&\leq CN^{2}\eta^{2}t.
\end{align}
Inserting this into \eqref{eq:EY} we obtain
\begin{align}
    1_{\mc{E}}\mbb{E}_{Y}\left[\sum_{n=1}^{N}f_{\eta}(z_{n})g_{z_{0}}(z_{n})\right]&\leq CN^{3}\eta^{2}\|g_{z_{0}}\|_{1}\leq CN^{2}\eta^{2},
\end{align}
and so
\begin{align*}
    \mbb{E}_{B}\left[\sum_{n=1}^{N}f_{\eta}(z_{n})g_{z_{0}}(z_{n})\right]&\leq C\mbb{E}_{X}\left[1_{\mc{E}}\mbb{E}_{Y}\left[\sum_{n=1}^{N}f_{\eta}(z_{n})g_{z_{0}}(z_{n})\right]\right]\\
    &\leq CN^{2}\eta^{2},
\end{align*}
as claimed.

\subsection{Real Entries and Complex Eigenvalues}
Applying \eqref{eq:real-complex} we have
\begin{align}
    \mbb{E}_{B}\left[\sum_{n=1}^{N_{\mbb{C}}}f_{\eta}(z_{n})g_{z_{0}}(z_{n})\right]&=\left(\frac{N}{2\pi t}\right)^{3}\int_{\mbb{C}_{+}}g_{z_{0}}(z)\nonumber\\
    &\times\int_{0}^{\infty}\frac{2y\delta}{\sqrt{\delta^{2}+4y^{2}}}L(z,\delta)\mbb{E}_{\xi_{\delta,z}}\left[\mbb{E}_{Y'}\left[f_{\eta}(z)\det|B'_{z}|^{2}\right]\right]d\delta dz,
\end{align}
where $B'=X^{(V)}+\sqrt{\frac{Nt}{N-2}}Y'$ and $Y'$ is a real Ginibre matrix of size $N-2$. 

The function $f_{\eta}(z)$ takes the form
\begin{align*}
    f_{\eta}(z)&=\eta^{2}\tr\begin{pmatrix}|Z-z|^{2}+\eta^{2}&(Z^{T}-\bar{z})W^{T}\\W(Z-z)&|B'_{z}|^{2}+WW^{T}+\eta^{2}\end{pmatrix}^{-1}-1\\
    &=\eta^{2}\tr\left(|Z-z|^{2}-(Z^{T}-\bar{z})W^{T}(|B'_{z}|^{2}+WW^{T}+\eta^{2})^{-1}W(Z-z)+\eta^{2}\right)^{-1}\\
    &+\eta^{2}\tr\left(|B'_{z}|^{2}+WW^{T}-W(Z-z)(|Z-z|^{2}+\eta^{2})^{-1}(Z^{T}-\bar{z})W^{T}+\eta^{2}\right)^{-1}\\
    &-1.
\end{align*}
To simplify this expression we note that
\begin{align*}
    1-(Z-z)(|Z-z|^{2}+\eta^{2})^{-1}(Z^{T}-\bar{z})&=\eta^{2}(|Z^{T}-\bar{z}|^{2}+\eta^{2})^{-1},\\
    1-W^{T}(|B'_{z}|^{2}+WW^{T}+\eta^{2})^{-1}W&=(1+W^{T}(|B'_{z}|^{2}+\eta^{2})^{-1}W)^{-1},
\end{align*}
which gives us
\begin{align*}
    f_{\eta}(z)&=\eta^{2}\tr\left((Z^{T}-\bar{z})(1+W^{T}(|B'_{z}|^{2}+\eta^{2})^{-1}W)^{-1}(Z-z)+\eta^{2}\right)^{-1}\\
    &+\eta^{2}\tr\left(|B'_{z}|^{2}+\eta^{2}W(|Z^{T}-\bar{z}|^{2}+\eta^{2})^{-1}W^{T}+\eta^{2}\right)^{-1}\\
    &-1\\
    &\leq\eta^{2}\tr\left(|Z-z|^{2}+\eta^{2}\right)^{-1}+\eta^{2}\tr(|B'_{z}|^{2}+\eta^{2})^{-1}-1.
\end{align*}
Now we note that there is a unitary $U\in\mbb{U}(2)$ such that
\begin{align*}
    |Z-z|^{2}+\eta^{2}&=U\begin{pmatrix}\eta^{2}&0\\0&\delta^{2}+4y^{2}+\eta^{2}\end{pmatrix}U^{*},
\end{align*}
and so
\begin{align}
    f_{\eta}(z)&\leq\frac{\eta^{2}}{\delta^{2}+4y^{2}+\eta^{2}}+\eta^{2}\tr|B'_{z}|^{-2},
\end{align}
where we have again used the crude bound $\tr(|B'_{z}|^{2}+\eta^{2})^{-1}\leq\tr|B'_{z}|^{-2}$.

Using anti-commuting variables one can show that the expectation of $\det|B'_{z}|^{2}$ is the same as in the complex case (this is only true for the second moment):
\begin{align}
    \mbb{E}_{Y'}\left[\det\left(|B'_{z}|^{2}+\eta^{2}\right)\right]&=\frac{N}{\pi t}\int_{\mbb{C}}e^{-\frac{N}{t}|q-\eta|^{2}}\det\left(|q|^{2}+|X^{(V)}_{z}|^{2}\right)dq.
\end{align}
Therefore we can repeat the analysis from the previous subsection to obtain
\begin{align}
    \mbb{E}_{Y'}\left[f_{\eta}(z)\det|B'_{z}|^{2}\right]&\leq C\sqrt{Nt}\left(\frac{\eta^{2}}{\delta^{2}+4y^{2}+\eta^{2}}+N^{2}\eta^{2}\right)e^{-N\phi^{(V)}_{z}}.
\end{align}

Let 
\begin{align}
    \wh{L}(\delta,z)&=e^{-N\phi_{z}}L(\delta,z).
\end{align}
In \cite[Lemma 6.9]{osman_bulk_2024} we obtained the asymptotics
\begin{align}
    \wh{L}(\delta,z)&\leq e^{-\frac{CN}{t}\delta^{2}},\quad\delta>C\|X\|,\\
    \wh{L}(\delta,z)&\leq e^{-C\log^{2}N},\quad\frac{\log N}{\sqrt{Nt}}<\delta<C\|X\|,
\end{align}
and
\begin{align}
    \wh{L}(\delta,z)&=\left[1+O\left(\frac{\log^{2}N}{\sqrt{Nt}}\right)\right]\frac{2^{5/2}\pi^{3/2}}{N^{3/2}\Tr{H^{2}_{z}}^{1/2}\Tr{H_{z}H_{\bar{z}}}}\nonumber\\
    &\times\left[\exp\left\{-\frac{N\wt{\sigma}_{z}}{2}\delta^{2}\right\}+O(N^{-D})\right],\quad \delta<\frac{\log N}{\sqrt{Nt}}.\label{eq:LAsymp}
\end{align}
for any $D>0$. As stated in the appendix of \cite{osman_bulk_2024}, this is uniform in $y=\Im z\geq0$. By Cramer's rule and the bound $|\eta^{(V)}_{z}-\eta_{z}|<CN^{-1}$ we have
\begin{align}
    e^{N(\phi_{z}-\phi^{(V)}_{z})}&\leq\frac{C}{t^{4}},
\end{align}
uniformly in $\delta$. Together with the asymptotics of $\wh{L}$ we conclude that we can restrict to $\delta<\frac{\log N}{\sqrt{N}}$. From \cite[Lemma 6.11]{osman_bulk_2024} we have
\begin{align*}
    \frac{\det\left[(\eta^{(V)}_{z})^{2}+|X^{(V)}_{z}|^{2}\right]}{\det\left(\eta_{z}^{2}+|X_{z}|^{2}\right)}&=\left|\det1_{2}\otimes V^{T}G^{(V)}_{z}1_{2}\otimes V\right|\\
    &=\left[1+O\left(\frac{\log N}{\sqrt{Nt^{3}}}\right)\right]t^{4}\Tr{H^{2}_{z}}\Tr{H_{z}H_{\bar{z}}}\sigma_{z}\wt{\sigma}_{z},
\end{align*}
with probability $1-O(e^{-C\log^{2}N})$, uniformly in $\delta<\frac{\log N}{\sqrt{N}}$. Changing variables $\delta\mapsto\delta/\sqrt{N\wt{\sigma}_{z}}$ we find (where $z=x+iy$)
\begin{align*}
    1_{\mc{E}}\mbb{E}_{Y}\left[\sum_{n=1}^{N_{\mbb{C}}}f_{\eta}(z_{n})g_{z_{0}}(z_{n})\right]&\leq CN\int_{\mbb{C}_{+}}g_{z_{0}}(z)\int_{0}^{\infty}\left(N^{2}\eta^{2}+\frac{N\eta^{2}}{\delta^{2}+4Ny^{2}+N\eta^{2}}\right)\frac{\sqrt{N}y\delta}{\sqrt{\delta^{2}+4Ny^{2}}}e^{-\frac{1}{2}\delta^{2}}d\delta dz\\
    &=CN\int_{\mbb{C}_{+}}g_{z_{0}}(z)\sqrt{N}ye^{2Ny^{2}}\int_{2\sqrt{N}y}^{\infty}\left(N^{2}\eta^{2}+\frac{N\eta^{2}}{\delta^{2}+N\eta^{2}}\right)e^{-\frac{1}{2}\delta^{2}}d\delta dz\\
    &\leq CN^{2}\eta^{2}.
\end{align*}

\subsection{Real Entries and Real Eigenvalues}
Applying \eqref{eq:real-real} we have
\begin{align}
    \mbb{E}_{Y}\left[\sum_{n=1}^{N_{\mbb{R}}}f_{\eta}(u_{n})g_{u_{0}}(u_{n})\right]&=\frac{N}{4\pi t}\int_{\mbb{R}}g_{u_{0}}(u)K(u)\mbb{E}_{\nu_{u}}\left[\mbb{E}_{Y'}\left[f_{\eta}(u)\det|B'_{u}|\right]\right]du,
\end{align}
where $B'=X^{(\mbf{v})}+\sqrt{\frac{Nt}{N-1}}Y'$ and $Y'$ is a real Ginibre matrix of size $N-1$.

In these coordinates $f_{\eta}$ takes the form
\begin{align}
    f_{\eta}(u)&=\eta^{2}\tr\left(|B'_{u}|^{2}+\mbf{w}\mbf{w}^{T}+\eta^{2}\right)^{-1}\nonumber\\
    &\leq\eta^{2}\tr(|B'_{u}|^{2}+\eta^{2})^{-1}.
\end{align}
This time we cannot afford to take the crude bound $\tr(|B'_{u}|^{2}+\eta^{2})^{-1}\leq\tr|B'_{u}|^{-2}$ due to a logarithmic singularity. Instead we use a different crude bound $\det|B|\leq\det^{1/2}\left(|B|^{2}+\eta^{2}\right)$:
\begin{align}
    \mbb{E}_{Y'}\left[f_{\eta}(u)\det|B'_{u}|\right]&\leq\eta^{2}\mbb{E}_{Y'}\left[\tr(|B'_{u}|^{2}+\eta^{2})^{-1}\det^{1/2}\left(|B'_{u}|^{2}+\eta^{2}\right)\right]\nonumber\\
    &=-2\eta^{2}\mbb{E}_{Y'}\left[\det\left(|B'_{u}|^{2}+\eta^{2}\right)\frac{\partial}{\partial\eta^{2}}\det^{-1/2}\left(|B'_{u}|^{2}+\eta^{2}\right)\right]\nonumber\\
    &=:F(\eta).
\end{align}

For convencience we drop the superscript $(\mbf{v})$ for now and restore it at the end; the analysis relies only on properties of singular values of $X_{u}$ which also hold for $X^{(\mbf{v})}_{u}$ by interlacing. Following the same steps as in \Cref{sec4} we can obtain the formula
\begin{align}
    F(\eta)&=\frac{N^{2}\eta^{2}}{2\pi t^{2}}\int_{0}^{\infty}\frac{r}{(1+r)^{3}}e^{-\frac{N}{2t}\eta^{2}r}\psi(r)dr,
\end{align}
where
\begin{align}
    \psi(r)&=\left(\frac{Nr}{2\pi t(1+r)}\right)^{N/2-1}\int_{S^{N-1}}\exp\left\{-\frac{Nr}{2t(1+r)}\mbf{u}^{T}|X_{u}|^{2}\mbf{u}\right\}I(r,\mbf{u})d_{H}\mbf{u}
\end{align}
and
\begin{align}
    I(r,\mbf{u})&=\frac{N}{\pi t}\int_{\mbb{C}}e^{-\frac{N}{t}|q-\eta|^{2}}\det\left(|q|^{2}+|X_{u}|^{2}\right)\cdot\left[1+rq(\bar{q}+\eta(1+r))\mbf{u}^{T}H_{u}(|q|)\mbf{u}\right]dq.
\end{align}
Changing variable $q=\sigma e^{i\theta}$ we obtain
\begin{align}
    I(r,\mbf{u})&=\frac{2N}{t}e^{-\frac{N}{t}\eta^{2}}\int_{0}^{\infty}e^{-N\phi_{u}(\sigma)}\left[\left(1+r\sigma^{2}\mbf{u}^{T}H_{u}(\sigma)\mbf{u}\right)I_{0}\left(\frac{2N\eta\sigma}{t}\right)\right.\nonumber\\&\left.+\eta r(1+r)\sigma\mbf{u}^{T}H_{u}(\sigma)\mbf{u}I_{1}\left(\frac{2N\eta\sigma}{t}\right)\right]\sigma d\sigma,
\end{align}
where $I_{m}(x)=\frac{1}{\pi}\int_{0}^{\pi}e^{x\cos\theta}\cos m\theta d\theta,\,m\in\mbb{N}$ is the modified Bessel function of the first kind with asymptotics
\begin{align}
    I_{m}(x)&\sim\begin{cases}\left(\frac{x}{2}\right)^{m}&\quad x\to0\\
    \frac{e^{x}}{\sqrt{2\pi x}}&\quad x\to\infty
    \end{cases}.
\end{align}
Since $\eta=N^{-1-\epsilon}$, the growth of $I_{m}(N\eta\sigma/t)$ is dominated by the decay of $\phi_{u}(\sigma)$. Arguing as in \Cref{sec4} we obtain the bound
\begin{align*}
    I(r,\mbf{u})&\leq C\sqrt{Nt}(1+r)\left(1+\frac{N\eta^{2} r}{t}\right)e^{-N\phi_{u}}
\end{align*}
where we have also used $\mbf{u}^{T}H_{u}\mbf{u}\leq Ct^{-2}$.

Inserting this bound into $\psi(r)$ and integrating over $\mbf{u}\in S^{N-1}$ using the duality formula we obtain
\begin{align*}
    \frac{\psi(r)}{1+r}&\leq C\sqrt{Nt}e^{-\frac{N}{2}\phi_{u}}\cdot\left(1+\frac{N\eta^{2} r}{t}\right)e^{-\frac{N}{2}\wt{\phi}_{u}(r)}h(r),
\end{align*}
where $\wt{\phi}_{u}(r)$ and $\eta_{u}(r)$ were defined in \eqref{eq:phiTilde} and \eqref{eq:eta_r} respectively and
\begin{align}
    h(r)&=\int_{-\infty}^{\infty}e^{\frac{iNrp}{2t(1+r)}}\det^{-1/2}\left[1+ipH_{u}(\eta_{u}(r))\right]dp.
\end{align}
Inserting this into the expression for $F(\eta)$ we obtain the bound
\begin{align}
    F(\eta)&\leq\frac{CN^{5/2}\eta^{2}}{t^{3/2}}e^{-\frac{N}{2}\phi_{u}}\int_{0}^{\infty}\frac{r}{(1+r)^{2}}\left(1+\frac{N\eta^{2} r}{t}\right)e^{-\frac{N}{2}\left[\frac{\eta^{2}r}{t}+\wt{\phi}_{u}(r)\right]}h(r)dr.
\end{align}
We can treat this in the same way we treated $D_{2m-1}(u)$ to obtain
\begin{align}
    F(\eta)&\leq CN^{2}\eta^{2}e^{-\frac{N}{2}\phi_{u}}\int_{1/\delta}^{\infty}\frac{1}{r}\left(1+\frac{N\eta^{2}r}{t}\right)e^{-\frac{N}{2}\left(\eta^{2}r/t+Ct/r\right)}dr\nonumber\\
    &\leq CN^{2}\eta^{2}|\log N\eta|\cdot e^{-\frac{N}{2}\phi_{u}}.
\end{align}
Restoring the superscript $(\mbf{v})$ we have thus found
\begin{align}
    \mbb{E}_{Y'}\left[f_{\eta}(u)\det|B'_{u}|\right]&\leq CN^{2}\eta^{2}|\log N\eta|\cdot e^{-\frac{N}{2}\phi^{(\mbf{v})}_{u}}.
\end{align}
The rest of the proof is the same as in the complex case. From \cite[Lemma 6.2, Lemma 6.4]{osman_bulk_2024} we have
\begin{align}
    K(u)&=\left[1+O\left(\frac{\log^{3}N}{\sqrt{Nt}}\right)\right]\sqrt{\frac{4\pi}{N\Tr{H^{2}_{u}}}}e^{\frac{N}{2}\phi_{u}}\leq\frac{Ct^{3/2}}{N^{1/2}}e^{\frac{N}{2}\phi_{u}},\\
    e^{\frac{N}{2}\left[\phi_{u}-\phi^{(\mbf{v})}_{u}\right]}&=\left[1+O\left(\frac{\log^{2}N}{\sqrt{Nt}}\right)\right]\det^{1/2}\left[(1_{2}\otimes\mbf{v}^{T})G_{u}(1_{2}\otimes\mbf{v})\right],
\end{align}
and
\begin{align}
    \mbb{E}_{\nu_{u}}\left[\det^{1/2}\left[(1_{2}\otimes\mbf{v}^{T})G_{u}(1_{2}\otimes\mbf{v})\right]\right]&=\left[1+O\left(\frac{\log^{2}N}{\sqrt{Nt^{3}}}\right)\right]\sqrt{t^{2}\Tr{H^{2}_{u}}\sigma_{u}}\leq Ct^{-1/2}.
\end{align}
Putting everything together we obtain
\begin{align}
    1_{\mc{E}}\mbb{E}_{Y}\left[\sum_{n=1}^{N_{\mbb{R}}}f_{\eta}(u_{n})g_{u_{0}}(u_{n})\right]&\leq CN^{2}\eta^{2}|\log N\eta|\cdot N^{1/2}\|g_{u_{0}}\|_{1}\nonumber\\
    &\leq CN^{2}\eta^{2}|\log N\eta|.
\end{align}

\section{Proof of \Cref{thm2}}\label{sec7}
We write the details in the case $l=1$, i.e. we fix a $\mbf{q}\in\mbf{C}^{N}$ and consider
\begin{align}
    \mc{L}_{\theta}(z_{0},\mbf{q})&=\sum_{n=1}^{N}\theta\left(\sqrt{N}(z_{n}-z_{0}),N|\mbf{q}^{*}\mbf{r}_{n}|^{2}\right),
\end{align}
where $\theta:\mbb{C}\times\mbb{R}_{+}$ is such that $\theta(z,x)=0$ for $|z|>C$ and
\begin{align}
    \left|\frac{\partial^{p}}{\partial x^{p}}\theta(z,x)\right|&\leq(1+x)^{C},\quad p=0,...,5.
\end{align}
The extension to $l>1$ does not pose any additional difficulties.

We fix small, positive constants $\epsilon,\xi,\zeta$ such that $\epsilon>\zeta>\xi$ and set $\eta=N^{-1-\epsilon}$. At the end we will choose $\epsilon,\xi,\zeta$ small enough (depending on $\theta$) so that all the error terms are $O(N^{-\delta})$ for some $\delta>0$. Let $\mbf{v}_{1}(z),...,\mbf{v}_{N}(z)$ denote the singular vectors of $X_{z}$. We define the events
\begin{align}
    \mc{E}_{1}&=\left\{\left|\{n:|z_{n}-z_{0}|<rN^{-1/2}\}\right|<N^{\xi}\right\},\\
    \mc{E}_{2}&=\left\{\sup_{|z-z_{0}|<rN^{-1/2}}|\mbf{q}^{*}\mbf{v}_{N}(z)|<N^{-1/2+\xi/2}\|\mbf{q}\|\right\},\\
    \mc{E}_{3}&=\left\{\sup_{|z-z_{0}|<rN^{-1/2}}\sup_{|E|<N^{\zeta}\eta}|\mbf{q}^{*}G_{z}(E+i\eta)\mbf{q}|<\frac{N^{\xi+\zeta}}{N\eta}\right\},\\
    \mc{E}_{4}&=\left\{\inf_{|z_{n}-z_{0}|<rN^{-1/2}}s_{N-1}(z_{n})>N^{\epsilon/2}\eta\right\}.
\end{align}
By local laws and \Cref{lem:aPriori} (and a net argument), $\mc{E}_{1},\mc{E}_{2}$ and $\mc{E}_{3}$ hold with probability $1-N^{-D}$ for any $D>0$. Since $|\mc{L}(z_{0},\mbf{q})|<CN(1+N\|\mbf{q}\|^{2})^{C}$ by the assumption on $\theta$, we have
\begin{align}
    \mbb{E}\left[\mc{L}(z_{0},\mbf{q})\right]&=\mbb{E}\left[\mc{L}(z_{0},\mbf{q})1_{\mc{E}_{1}}1_{\mc{E}_{2}}1_{\mc{E}_{3}}\right]+O(N^{-D}).
\end{align}
By the assumption on the support of $\theta$, on the event $\mc{E}_{1}$ only $N^{\xi}$ terms in the sum are non-zero and so by \Cref{thm1} we have
\begin{align}
    \mbb{E}\left[\mc{L}(z_{0},\mbf{q})\right]&=\mbb{E}\left[\mc{L}(z_{0},\mbf{q})1_{\mc{E}_{1}}1_{\mc{E}_{2}}1_{\mc{E}_{3}}1_{\mc{E}_{4}}\right]+O(N^{C\xi-\epsilon/4}).\label{eq:mcL}
\end{align}
Now we define
\begin{align}
    \wh{\mc{L}}(z_{0},\mbf{q})&=\sum_{n=1}^{N}\theta\left(\sqrt{N}(z_{n}-z_{0}),\frac{N}{\pi}\int_{-N^{\zeta}\eta}^{N^{\zeta}\eta}\Im \wt{\mbf{q}}^{*}G_{z_{n}}(E+i\eta)\wt{\mbf{q}} dE\right),
\end{align}
where $\wt{\mbf{q}}=(0,\mbf{q})^{T}\in\mbb{C}^{2N}$. By exactly the same arguments we have
\begin{align}
    \mbb{E}\left[\wh{\mc{L}}(z_{0},\mbf{q})\right]&=\mbb{E}\left[\wh{\mc{L}}(z_{0},\mbf{q})1_{\mc{E}_{1}}1_{\mc{E}_{2}}1_{\mc{E}_{3}}1_{\mc{E}_{4}}\right]+O\left(N^{C(\xi+2\zeta)-\epsilon/4}\right).\label{eq:mcLHat}
\end{align}
We can now compare $\mc{L}$ and $\wh{\mc{L}}$. Using the spectral decomposition of $G_{z_{n}}$ we have
\begin{align}
    \frac{N}{\pi}\int_{-N^{\zeta}\eta}^{N^{\zeta}\eta}\Im\wt{\mbf{q}}^{*}G_{z_{n}}(E+i\eta)\wt{\mbf{q}}dE&=\frac{N|\mbf{q}^{*}\mbf{r}_{n}|^{2}}{\pi}\int_{-N^{\zeta}\eta}^{N^{\zeta}\eta}\frac{\eta}{E^{2}+\eta^{2}}dE\label{eq:I1}\\
    &+\sum_{m=1}^{N-1}\frac{N|\mbf{q}^{*}\mbf{v}_{m}|^{2}}{\pi}\int_{-N^{\zeta}\eta}^{N^{\zeta}\eta}\frac{\eta}{(s_{m}(z_{n})-E)^{2}+\eta^{2}}dE\label{eq:I2}\\
    &+\sum_{m=1}^{N-1}\frac{N|\mbf{q}^{*}\mbf{v}_{m}|^{2}}{\pi}\int_{-N^{\zeta}\eta}^{N^{\zeta}\eta}\frac{\eta}{(s_{m}(z_{n})+E)^{2}+\eta^{2}}dE.\label{eq:I3}
\end{align}
A direct calculation shows that the integral in \eqref{eq:I1} is $1+O(N^{-\zeta})$. On the event $\mc{E}_{3}\cap\mc{E}_{4}$, we have
\begin{align*}
    \eqref{eq:I2}&\leq CN^{1+\zeta}\sum_{m=1}^{N-1}\frac{\eta^{2}|\mbf{q}^{*}\mbf{v}_{m}|^{2}}{s_{m}^{2}(z_{n})+N^{\epsilon}\eta^{2}}\\
    &\leq CN^{1+\zeta-\epsilon}\sum_{m=1}^{N-1}\frac{(N^{\epsilon/2}\eta)^{2}|\mbf{q}^{*}\mbf{v}_{m}|^{2}}{s^{2}_{m}(z_{n})+(N^{\epsilon/2}\eta)^{2}}\\
    &\leq CN^{1+\zeta-\epsilon}\cdot(N^{\epsilon/2}\eta)\Im\wt{\mbf{q}}^{*}G_{z_{n}}(iN^{\epsilon/2}\eta)\wt{\mbf{q}}\\
    &\leq CN^{\xi+\zeta-\epsilon},
\end{align*}
and likewise for \eqref{eq:I3}. Thus we find
\begin{align}
    1_{\mc{E}_{2}}1_{\mc{E}_{3}}1_{\mc{E}_{4}}\frac{N}{\pi}\int_{-N^{\zeta}\eta}^{N^{\zeta}\eta}\Im\wt{\mbf{q}}^{*}G_{z_{n}}(E+i\eta)\wt{\mbf{q}}dE&=N|\mbf{q}^{*}\mbf{r}_{n}|^{2}+O(N^{\xi-\zeta})+O(N^{\xi+\zeta-\epsilon}).
\end{align}
By Taylor expansion using the assumption on $\theta$ we obtain
\begin{align}
    \mbb{E}\left[\left(\mc{L}(z_{0},\mbf{q})-\wh{\mc{L}}(z_{0},\mbf{q})\right)1_{\mc{E}_{1}}1_{\mc{E}_{2}}1_{\mc{E}_{3}}1_{\mc{E}_{4}}\right]&=O(N^{-\delta}),
\end{align}
for some $\delta>0$ depending on $\theta$, as long as $\xi$ is sufficiently smaller than $\zeta$ which in turn is sufficiently smaller than $\epsilon$. By \eqref{eq:mcL}, \eqref{eq:mcLHat} and the triangle inequality we have
\begin{align}
    \left|\mbb{E}\left[\mc{L}(z_{0},\mbf{q})\right]-\mbb{E}\left[\wh{\mc{L}}(z_{0},\mbf{q})\right]\right|&\leq N^{-\delta}.
\end{align}

Now we note that
\begin{align*}
    f:z&\mapsto\theta\left(\sqrt{N}(z-z_{0}),\frac{N}{\pi}\int_{-N^{\zeta}\eta}^{N^{\zeta}\eta}\Im\wt{\mbf{q}}^{*}G_{z}(E+i\eta)\wt{\mbf{q}}dE\right)
\end{align*}
is $C^{2}$ and so we can apply Girko's formula to $\wh{\mc{L}}$:
\begin{align}
    \wh{\mc{L}}(z_{0},\mbf{q})&=-\frac{1}{4\pi}\int_{\mbb{C}}\Delta_{z}f(z)\int_{0}^{\infty}\Im\tr G_{z}(i\sigma)d\sigma dz.
\end{align}
The rest of the proof follows the same pattern as the proof of \Cref{lem:comp}; the necessary bounds on $f$ follow from \Cref{lem:aPriori}.

In the real case, we follow Tao--Vu \cite{tao_random_2015}. Using the level repulsion bound \cite[Lemma 39]{tao_random_2015}
\begin{align}
    P\left(\left|\{n:|z_{n}-u_{0}|<N^{-1/2-\tau}\}\right|>1\right)&\leq CN^{-\tau}
\end{align}
we can replace $\theta$ with $\wt{\theta}$ which is supported in $\{x+iy:|x-u_{0}|<rN^{-1/2},|y|<N^{-1/2-\tau}\}\times\mbb{R}^{l}$:
\begin{align*}
    &\mbb{E}\left[\sum_{n=1}^{N_{\mbb{R}}}\theta(\sqrt{N}(u_{n}-u_{0}),N|\mbf{q}_{1}^{T}\mbf{r}^{\mbb{R}}_{n}|^{2},...,N|\mbf{q}_{l}^{T}\mbf{r}^{\mbb{R}}_{n}|^{2})\right]=\\
    &\mbb{E}\left[\sum_{n=1}^{N}\wt{\theta}(\sqrt{N}(z_{n}-u_{0}),N|\mbf{q}_{1}^{T}\mbf{r}_{n}|^{2},...,N|\mbf{q}_{l}^{T}\mbf{r}_{n}|^{2})\right]+O(N^{-\delta}).
\end{align*}
The new statistic is a sum over all eigenvalues and so we can apply Girko's formula. The extra powers of $N^{\tau}$ from the derivatives of $\wt{\theta}$ can be absorbed in the $O(N^{-\delta})$ term for sufficiently small $\tau$. We can follow a similar procedure for complex eigenvalues, where we replace functions on the upper half-plane with functions on the whole plane.

\section{Proof of \Cref{thm3}}\label{sec8}
Consider first real eigenvalues of real matrices. In view of \Cref{thm2}, we only need to evaluate 
\begin{align}
    \mbb{E}_{B}\left[\sum_{n=1}^{N}g_{z_{0}}(z_{n})\theta\left(N|\mbf{q}_{1}^{T}\mbf{r}_{n}|^{2},...,N\mbf{q}_{l}^{T}\mbf{r}_{n}|^{2}\right)\right]
\end{align}
where $B=X+\sqrt{t}Y$ is Gauss-divisible and $\theta(\mbf{x})$ is a monomial
\begin{align}
    \theta(\mbf{x})&=\prod_{j=1}^{l}x_{j}^{m_{j}}.
\end{align}
Since $\|\theta\|_{\infty}\leq N^{C}$ for some fixed $C$, we can restrict to the event $\mc{E}$ that $X$ satisfies the local laws in \Cref{prop:ll1,prop:ll2}.

The partial Schur decomposition gives us the formula
\begin{align}
    \mbb{E}_{Y}\left[\sum_{n=1}^{N}g_{u_{0}}(u_{n})\theta(\mbf{r}_{n})\right]&=\frac{N}{4\pi t}\int_{-\infty}^{\infty}g_{u_{0}}(u)K_{\mbb{R}}(u)\mbb{E}_{\nu_{u}}\left[\theta(\mbf{v})\mbb{E}_{Y'}\left[\det|B'_{u}|\right]\right]du.
\end{align}
The only difference between this and the integrals evaluated in \Cref{sec4,sec6} is the presence of $\theta(\mbf{v})$. Since on the event $\mc{E}$ we have the bound $\|\theta\|_{\infty}<N^{C\xi}$, we can repeat the same asymptotic analysis to obtain
\begin{align}
    \mbb{E}_{Y}\left[\sum_{n=1}^{N}g_{u_{0}}(u_{n})\theta(\mbf{r}_{n})\right]&=\left[1+O\left(\frac{\log^{2}N}{\sqrt{Nt}}\right)\right]\sqrt{\frac{N}{2\pi}}\int_{\mbb{R}}\sqrt{\sigma_{u}}g_{u_{0}}(u)\mbb{E}_{\nu_{u}}\left[\theta(\mbf{v})\right]du.
\end{align}
From \cite[Lemma 8.1]{maltsev_bulk_2023} we have $\sigma_{u}=1+O(t)$. Using the duality formula we can rewrite the expectation value as
\begin{align}
    \mbb{E}_{\nu_{u}}\left[\theta(\mbf{v})\right]&=\frac{e^{\frac{N}{2}\phi_{u}}}{K_{\mbb{R}}(u)}\int_{-\infty}^{\infty}e^{\frac{iNp}{2t}}\det^{-1/2}\left(1+ipH_{u}\right)h(p)dp,
\end{align}
where
\begin{align}
    h(p)&=\frac{1}{(2\pi)^{N/2}}\int_{\mbb{R}^{N}}e^{-\frac{1}{2}\|\mbf{x}\|^{2}}\prod_{j=1}^{l}\left(t\mbf{x}^{T}\sqrt{H_{u}(w_{p})}\mbf{q}_{j}\mbf{q}_{j}^{T}\sqrt{H_{u}(w_{p})}\mbf{x}\right)^{m_{j}}d\mbf{x},
\end{align}
and $w_{p}=\sqrt{\eta_{u}^{2}+ip}$. By Wick's theorem $h(p)$ is a sum of products of traces. Each term can be generated by choosing a permutation $\sigma$ of $m=\sum_{j=1}^{l}m_{j}$ elements, taking the trace of the product of $t\sqrt{H_{u}(w_{p})}\mbf{q}_{j}\mbf{q}_{j}^{T}\sqrt{H_{u}(w_{p})}$ in each cycle of $\sigma$, and then taking the product over the cycles. For example, for the permutation
\begin{align*}
    \sigma=(1,\cdots,m_{1})\cdots(m-m_{l}+1,\cdots,m)
\end{align*}
we have the term
\begin{align*}
    &t^{m}\tr\left(\sqrt{H_{u}(w_{p})}\mbf{q}_{1}\mbf{q}^{T}_{1}\sqrt{H_{u}(w_{p})}\right)^{m_{1}}\cdots\tr\left(\sqrt{H_{u}(w_{p})}\mbf{q}_{l}\mbf{q}_{l}^{T}\sqrt{H_{u}(w_{p})}\right)^{m_{l}}\\
    &=t^{m}\left(\mbf{q}_{1}^{T}H_{u}(w_{p})\mbf{q}_{1}\right)^{m_{1}}\cdots\left(\mbf{q}_{l}^{T}H_{u}(w_{p})\mbf{q}_{l}\right)^{m_{l}}.
\end{align*}
Each term can be estimated by the isotropic local law:
\begin{align*}
    t\left|\mbf{q}^{T}H_{u}(w_{p})\mbf{q}-\Tr{H_{u}(w_{p})}\|\mbf{q}\|^{2}\right|&\prec\frac{\|\mbf{q}\|^{2}}{\sqrt{Nt}}.
\end{align*}
The bound $h(p)\leq CN^{C\xi}$ allows us to restrict $p$ to the interval $|p|<\sqrt{\frac{t^{3}}{N}}\log N$, in which
\begin{align*}
    t\Tr{H_{u}(w_{p})}&=1+O\left(\frac{\log N}{\sqrt{Nt}}\right).
\end{align*}
and so
\begin{align*}
    h(p)&=\left[1+O\left(\frac{\log N}{\sqrt{Nt}}\right)\right]\frac{1}{(2\pi)^{N/2}}\int_{\mbb{R}^{N}}e^{-\frac{1}{2}\|\mbf{x}\|^{2}}\prod_{j=1}^{l}|\mbf{q}_{j}^{T}\mbf{x}|^{2m_{j}}d\mbf{x}\\
    &=\left[1+O\left(\frac{\log N}{\sqrt{Nt}}\right)\right]\int_{\mbb{R}^{l}_{+}}\theta(\mbf{x})\rho_{\mbf{q}_{1},...,\mbf{q}_{l}}(\mbf{x})d\mbf{x},
\end{align*}
where $\rho_{\mbf{q}_{1},...,\mbf{q}_{l}}$ is the density of $(|\mbf{q}^{T}_{1}\mbf{p}|^{2},...,|\mbf{q}^{T}_{l}\mbf{p}|^{2})$ for a standard Gaussian vector $\mbf{p}\in\mbb{R}^{N}$. Complex matrices can be treated in exactly the same way.

Now consider complex eigenvalues of real matrices. Applying \eqref{eq:real-complex} we find
\begin{align}
    \mbb{E}_{Y}\left[\sum_{n=1}^{N_{\mbb{C}}}g_{z_{0}}(z)\theta\left(N|\mbf{q}^{*}_{1}\mbf{r}_{n}|^{2},...,N|\mbf{q}^{*}_{l}\mbf{r}_{n}|^{2}\right)\right]&=\left(\frac{N}{2\pi t}\right)^{3}\int_{\mbb{C}_{+}}g_{z_{0}}(z)\int_{0}^{\infty}\frac{2y\delta}{\sqrt{\delta^{2}+4y^{2}}}\nonumber\\
    &\times L(\delta,y)\mbb{E}_{\xi_{\delta,z}}\left[\wt{\theta}(V)\mbb{E}_{Y'}\left[\det|B'_{z}|^{2}\right]\right]d\delta dz,
\end{align}
where
\begin{align}
    \wt{\theta}(V)&=\prod_{j=1}^{l}\textup{vec}(V)^{T}Q_{j}\textup{vec}(V),\\
    Q_{j}&=\begin{pmatrix}\alpha^{2}(\mbf{a}_{j}\mbf{a}_{j}^{T}+\mbf{b}_{j}\mbf{b}_{j}^{T})&\alpha\beta(\mbf{a}_{j}\mbf{b}_{j}^{T}-\mbf{b}_{j}\mbf{a}_{j}^{T})\\
    \alpha\beta(\mbf{b}_{j}\mbf{a}_{j}^{T}-\mbf{a}_{j}\mbf{b}_{j}^{T})&\beta^{2}(\mbf{a}_{j}\mbf{a}_{j}^{T}+\mbf{b}_{j}\mbf{b}_{j}^{T})\end{pmatrix},
\end{align}
and we have defined $\mbf{q}_{j}=\mbf{a}_{j}+i\mbf{b}_{j}$ for $\mbf{a}_{j},\mbf{b}_{j}\in\mbb{R}^{N}$ and 
\begin{align}
    \alpha&=\sqrt{\frac{b}{b+c}},\quad\beta=\sqrt{\frac{c}{b+c}}.
\end{align}
Recall that $b$ and $c$ were defined in \eqref{eq:Z}.

As before, we can perform the same asymptotic analysis as in \Cref{sec6}, bounding the extra terms from $\theta$ by $N^{C}$ uniformly in all integration variables, and thus obtain
\begin{align}
    \mbb{E}_{Y}\left[\sum_{n=1}^{N_{\mbb{C}}}g_{z_{0}}(z)\theta(N|\mbf{q}_{1}^{*}\mbf{r}_{1}|^{2},...,N|\mbf{q}_{l}^{*}\mbf{r}_{n}|^{2})\right]&=\left[1+O\left(\frac{\log^{2}N}{\sqrt{Nt^{3}}}\right)\right]\frac{N}{\pi}\int_{\mbb{C}_{+}}g_{z_{0}}(z)\nonumber\\
    &\times\int_{0}^{\log N}\frac{2\sqrt{N\wt{\sigma}_{z}}y\delta}{\sqrt{\delta^{2}+4N\wt{\sigma}_{z}y^{2}}}e^{-\frac{1}{2}\delta^{2}}\mbb{E}_{\xi_{\delta,z}}\left[\wt{\theta}(V)\right]d\delta dz.
\end{align}
The expectation value can be evaluated using the duality formula:
\begin{align}
    \mbb{E}_{\xi_{\delta,z}}\left[\wt{\theta}(V)\right]&=\frac{1}{L(\delta,z)}\frac{e^{\frac{N}{t}\eta_{z}^{2}}}{\det^{1/2}\mc{M}_{0}}\int_{\mbb{M}^{sym}_{2}(\mbb{R})}e^{\frac{iN}{2t}\tr P}\det^{-1/2}\left(1+i\sqrt{\mc{M}_{0}}(P\otimes 1_{N})\sqrt{\mc{M}_{0}}\right)h(P)dP,
\end{align}
where
\begin{align}
    h(P)&=\frac{1}{(2\pi)^{N}}\int_{\mbb{R}^{2N}}e^{-\frac{1}{2}\|\mbf{x}\|^{2}}\prod_{j=1}^{l}\left(t\mbf{x}^{T}\sqrt{\mc{M}_{P}}Q_{j}\sqrt{\mc{M}_{P}}\mbf{x}\right)^{m_{j}}d\mbf{x},\quad P\in\mbb{M}^{sym}_{2}(\mbb{R}),
\end{align}
and
\begin{align}
    \mc{M}_{P}&=\left(\eta_{z}^{2}+iP\otimes1_{N}+|1_{2}\otimes X-Z^{T}\otimes1_{N}|^{2}\right)^{-1}.
\end{align}
We recall that a matrix with positive real part has a unique square root (after integration over $\mbf{x}$ only factors of $\mc{M}_{P}$ appear so the existence of the square root is unimportant). Using the bound $h(P)<N^{C}$ we can restrict $P$ to the region $\|P\|<\frac{t\log N}{\sqrt{N}}$ following the proof of \cite[Lemma 6.9]{osman_bulk_2024}. Wick's theorem  tells us that $h(P)$ consists of traces of products of $\sqrt{\mc{M}_{P}}Q_{j}\sqrt{\mc{M}_{P}}$ which themselves consist of inner products
\begin{align*}
    \mbf{c}^{T}(\mc{M}_{P})_{\mu\nu}\mbf{d},\quad \mu,\nu=1,2.
\end{align*}
We use the bound $\|P\|<\frac{t\log N}{\sqrt{N}}$ to reduce to the case $P=0$ using the fact that
\begin{align*}
    \tr(A+iB)^{-1}D&=\tr\frac{1-iA^{-1/2}BA^{-1/2}}{1+(A^{-1/2}BA^{-1/2})^{2}}A^{-1/2}DA^{-1/2}\\
    &=\left[1+O\left(\|A^{-1}\|\cdot\|B\|\right)\right]\tr A^{-1}D
\end{align*}
for positive $A,D$ and Hermitian $B$ when $\|A^{-1}\|\cdot\|B\|<1$. Let $\omega=(c/b)^{1/4}$ and
\begin{align}
    \mc{X}&=\begin{pmatrix}-i\eta&1_{2}\otimes X-Z^{T}\otimes1_{N}\\1_{2}\otimes X^{T}-Z\otimes1_{N}&-i\eta\end{pmatrix},\\
    S&=\frac{1}{\sqrt{2}}\begin{pmatrix}\omega&\omega\\-i/\omega&i/\omega\end{pmatrix},\\
    T&=\begin{pmatrix}1_{N}&0&0&0\\0&0&1_{N}&0\\0&1_{N}&0&0\\0&0&0&1_{N}\end{pmatrix}.
\end{align}
Then, with $E_{\mu\nu}$ denoting the matrix with 1 in the $(\mu,\nu)$ entry and zero elsewhere, we have (see \cite[Appendix B]{osman_bulk_2024})
\begin{align*}
    \mbf{c}^{T}(\mc{M}_{0})_{\mu\nu}\mbf{d}&=\frac{1}{i\eta}\tr\begin{pmatrix}1+\frac{i\delta v}{\sqrt{N\wt{\sigma}_{z}}}G_{z}E_{2,n}&-\frac{i\delta u}{\sqrt{N\wt{\sigma}_{z}}} G_{z}E_{2,n}\\\frac{i\delta u}{\sqrt{N\wt{\sigma}_{z}}}G_{\bar{z}}E_{2,n}&1-\frac{i\delta v}{\sqrt{N\wt{\sigma}_{z}}}G_{\bar{z}}E_{2,n}\end{pmatrix}^{-1}\begin{pmatrix}G_{z}&0\\0&G_{\bar{z}}\end{pmatrix}T\begin{pmatrix}0&0\\0&S^{-1}E_{\mu\nu}S\otimes\mbf{d}\mbf{c}^{T}\end{pmatrix}T,
\end{align*}
where $v=\frac{\delta}{2\sqrt{N\wt{\sigma}_{z}}y}$ and $u=\frac{\sqrt{\delta^{2}+4N\wt{\sigma}_{z}y^{2}}}{2\sqrt{N\wt{\sigma}_{z}}y}$. In the region $\delta<\log N$, we can make a series expansion of the first matrix in the above equation and estimate higher order terms by Cauchy-Schwarz and local laws, as in the proof of \cite[Lemma 6.9]{osman_bulk_2024}. The leading order terms will be
\begin{align}
    t\mbf{c}^{T}H_{z}\mbf{d},&\quad\mu=\nu=1,\\
    t\mbf{c}^{T}H_{\bar{z}}\mbf{d},&\quad\mu=\nu=2,\\
    \frac{it\omega^{2}}{2}\mbf{c}^{T}(H_{z}-H_{\bar{z}})\mbf{d},&\quad\mu=1,\nu=2,\\
    -\frac{it}{2\omega^{2}}\mbf{c}^{T}(H_{z}-H_{\bar{z}})\mbf{d},&\quad \mu=2,\nu=1.
\end{align}
The isotropic local law and the fact that $t\Tr{H_{z}}=t\Tr{H_{\bar{z}}}=1$ imply that the first two terms are $1+O(1/\sqrt{Nt})$ while the third and fourth are $O(1/\sqrt{Nt})$. We conclude that we can replace $\mc{M}_{P}$ with the identity in $h(P)$:
\begin{align*}
    h(P)&=\left[1+O\left(\frac{\log^{2}N}{\sqrt{Nt^{3}}}\right)\right]\frac{1}{(2\pi)^{N}}\int_{\mbb{R}^{2N}}e^{-\frac{1}{2}\|\mbf{x}\|^{2}}\prod_{j=1}^{l}\left(\mbf{x}^{T}Q_{j}\mbf{x}\right)^{m_{j}}d\mbf{x}.
\end{align*}
From this we observe that $\mbf{r}_{n}$ behaves as $\alpha\mbf{v}_{1}+i\beta\mbf{v}_{2}$ for two independent standard Gaussian vectors $\mbf{v}_{1},\mbf{v}_{2}\in\mbb{R}^{N}$. If $z_{0}\in\mbb{C}_{+}$, then $y>C>0$ in the support of $g_{z_{0}}$ (for sufficiently large $N$) and so when $\delta<\log N$ we have
\begin{align*}
    \alpha^{2}&=1+\frac{\delta}{\sqrt{\delta^{2}+4N\wt{\sigma}_{z}y^{2}}}=1+O\left(\frac{\log N}{\sqrt{N}}\right),\\
    \beta^{2}&=1-\frac{\delta}{\sqrt{\delta^{2}+4N\wt{\sigma}_{z}y^{2}}}=1+O\left(\frac{\log N}{\sqrt{N}}\right).
\end{align*}

If $z_{0}=u_{0}\in\mbb{R}$, then $y=O(N^{-1/2})$ in the support of $g_{u_{0}}$. We could argue in the same way as before but it is easier to see that the error terms are not singular as $y\to0$ from the alternative representation
\begin{align*}
    \mc{M}_{0}&=\begin{pmatrix}\eta_{z}^{2}+|X_{x}|^{2}+b^{2}&-bX_{x}+cX_{x}^{T}\\-bX_{x}^{T}+cX_{x}&\eta_{z}^{2}+|X_{x}|^{2}+c^{2}\end{pmatrix}^{-1}\\
    &=\left(1_{2}\otimes\sqrt{H_{x}}\right)\left[1+O\left(\frac{\log N}{\sqrt{Nt^{2}}}\right)\right]\left(1_{2}\otimes\sqrt{H_{x}}\right),
\end{align*}
where the second line follows because $|b|,|c|\leq C(|y|+\delta)$. Using the bound
\begin{align}
    \left|\eta_{z}-\eta_{x}\right|&\leq\frac{Ct}{\sqrt{N}}
\end{align}
from \cite[Lemma 3.5]{osman_bulk_2024} and the isotropic local law we again replace $\mc{M}_{P}$ with the identity. Now we cannot neglect the $\delta$ and $y$ dependence of $\alpha$ and $\beta$. From \cite[Lemma 8.1]{maltsev_bulk_2023} and the bound $|z-\bar{z}|<CN^{-1/2}$ in the support of $g_{u_{0}}$ we have
\begin{align*}
    \wt{\sigma}_{z}&=\left[1+O\left(\frac{1}{\sqrt{Nt^{3}}}\right)\right]\sigma_{z}\\
    &=1+O(t).
\end{align*}
The claim now follows after changing variable first to $y\mapsto y/\sqrt{\wt{\sigma}_{z}}$ and then to $y\mapsto y/\sqrt{N}$.

\paragraph{Acknowledgements}
This work was supported by the Royal Society, grant number \\RF/ERE210051.

\bibliographystyle{plain}
\bibliography{references.bib}

\end{document}